\documentclass[a4paper,11pt,reqno]{amsart}

\usepackage{url}
\usepackage[colorlinks=true, linkcolor=blue, citecolor=ForestGreen]{hyperref}
\usepackage{cite}

\usepackage{latexsym,amsmath,amssymb,amsfonts,amsthm}
\usepackage{mathrsfs}
\usepackage{mathtools}
\usepackage{dsfont}
\usepackage{bbm}
\usepackage{soul}     
\usepackage[utf8]{inputenc}
\usepackage{comment}
\usepackage{ulem}

\usepackage[usenames, dvipsnames]{color}
\usepackage[usenames, dvipsnames, cmyk]{xcolor}
\usepackage{graphicx}
\usepackage[scriptsize,hang,raggedright]{subfigure}
\usepackage{multirow}
\usepackage{enumerate}
\usepackage{enumitem}

\usepackage{a4wide}
\usepackage{epsfig,epstopdf}

\numberwithin{equation}{section}

\definecolor{grey}{rgb}{.7,.7,.7}
\definecolor{refkey}{gray}{.45}
\definecolor{labelkey}{gray}{.45}

\newtheorem{theorem}{Theorem}[section]
\newtheorem{prop}[theorem]{Proposition}
\newtheorem{lemma}[theorem]{Lemma}
\newtheorem{corol}[theorem]{Corollary}

\newtheorem*{mainthm}{Main Theorem}

\theoremstyle{remark}
\newtheorem{remark}[theorem]{Remark}
\theoremstyle{definition}
\newtheorem{defin}[theorem]{Definition}

\def\eps{\varepsilon}
\def\H{\mathscr H}

\def\L{\mathscr{L}}

\def\M{\mathsf{M}}
\def\P{\mathcal P}
\def\R{\mathbb R}
\def\S{\mathbb S}
\def\N{\mathbb N}

\def\AA{\mathsf{A}}
\def\AW{\mathcal{AW}}
\def\AP{\mathsf{AP}}

\def\Id{\mathrm{Id}}
\def\W{\mathcal{W}}

\def\bal{\begin{aligned}}
\def\eal{\end{aligned}}
\def\proofof#1{\begin{proof}[Proof of #1]}
\def\Chi#1{\hbox{{\large $\chi$}{\Large $_{_{#1}}$}}}

\def\XXint#1#2#3{{\setbox0=\hbox{$#1{#2#3}{\int}$} \vcenter{\vspace{-1pt}\hbox{$#2#3$}}\kern-.5\wd0}}

\def\XXiint#1#2#3{{\setbox0=\hbox{$#1{#2#3}{\iint}$} \vcenter{\vspace{-1pt}\hbox{$#2#3$}}\kern-0.5\wd0}}

\def\spt{{\rm spt}}
\newcommand{\norm}[1]{\left\lVert#1\right\rVert}
\newcommand{\weakstar}{\overset{\ast}{\rightharpoonup}}
\DeclareMathOperator{\diam}{diam}
\newcommand{\scal}[2]{\langle #1, #2 \rangle}
\def\bdry{\partial}

\newcommand{\res}{\mathop{\hbox{\vrule height 7pt width .5pt depth 0pt
\vrule height .5pt width 6pt depth 0pt}}\nolimits}

\newcounter{mt}

\def\maintheoremdeclaration#1{\stepcounter{mt}\newcounter{#1}\setcounter{#1}{\arabic{mt}}}

\maintheoremdeclaration{existence}
\maintheoremdeclaration{Linfty}
\maintheoremdeclaration{Unique}

\allowdisplaybreaks

\title[Maximizers of nonlocal interactions of {W}asserstein type]{Maximizers of nonlocal interactions of {W}asserstein Type}

\author{Almut Burchard}
\address[Almut Burchard]{Department of Mathematics, University of Toronto, ON, Canada}
\email{almut@math.toronto.edu}

\author{Davide Carazzato}
\address[Davide Carazzato]{Scuola Normale Superiore, Pisa, Italy}
\email{davide.carazzato@sns.it }

\author{Ihsan Topaloglu}
\address[Ihsan Topaloglu]{Department of Mathematics and Applied Mathematics, Virginia Commonwealth University, Richmond VA, United States}
\email{iatopaloglu@vcu.edu}

\date{\today}

\thanks{This is a post-peer-review, pre-copyedit version of an article published in ESAIM: Control, Optimisation and Calculus of Variations. The final
authenticated version is available online at: \url{https://doi.org/10.1051/cocv/2024068}.}

\begin{document}

\normalem

\begin{abstract}
We characterize the maximizers of a functional that involves the minimization of the Wasserstein distance between sets of equal volume. 
We prove that balls are the only maximizers by combining a symmetrization-by-reflection technique with the uniqueness of optimal transport plans. Further, in one dimension, we provide a sharp quantitative refinement of this maximality result.
\end{abstract}

\maketitle

\section{Introduction}\label{sec: intro}

In this paper we study a max-min problem involving the Wasserstein distance between two sets of equal volume. Specifically, for any $p>1$ we consider the following energy defined on subsets of $\R^N$:
\begin{equation}\label{eq:wasserstein}
    \W_p(E)\coloneqq \inf\Big\{W_p(\L^N\res E,\L^N\res F)\colon|F|=|E|,|E\cap F|=0\Big\},
\end{equation}
where $W_p(\mu_1,\mu_2)$ is the $p$-Wasserstein distance between two measures $\mu_1,\mu_2\in\M_+(\R^N)$ with $\mu_1(\R^N)=\mu_2(\R^N)<+\infty$. Here $\L^N$ denotes the Lebesgue measure in $\R^N$, and for any measurable set $E\subset \R^N$, we use the notation $|E|=\L^N(E)$.

The right hand side of \eqref{eq:wasserstein} defines a free
boundary problem associated with optimal 
partial transport. In these problems, 
given two measures $\mu_1,\mu_2$ and a mass $m\le\min\{\mu_1(\mathbb{R}^N),\mu_2(\mathbb{R}^N)\}$, the objective is to select portions 
$\tilde \mu_1,\tilde \mu_2$ of mass $m$ that minimize $W_p(\mu_1,\mu_2)$. Caffarelli and McCann~\cite{CMcC2010} introduce this problem, prove basic results on existence and uniqueness, and analyze the geometry of the solution when $p=2$. They show that for $i=1,2$, each of the optimal measure
$\tilde \mu_i$ agrees with $\mu_i$ on
some set $F_i$ (the {\em active regions}) and vanishes on the complement. 
A fundamental concern addressed in \cite{CMcC2010} is the
regularity of the free boundaries $\partial F_i$. Subsequent refinements of these regularity results can be found in~\cite{F2010,I2013}.

In the case of \eqref{eq:wasserstein},
the source for the partial transport problem is $\mu_1=\tilde \mu_1= \mathcal{L}^N \res E$, the mass is $m=\mathcal{L}^N(E)$, the target measure is $\mu_2=\mathcal{L}^N \res (\mathbb{R}^N\setminus E)$, and the active region for the target is~$F$. This belongs
to a class of problems where the Wasserstein distance is minimized among mutually singular measures that has  been investigated by Buttazzo, Carlier, and Laborde in \cite{BCL2020} for any $p\geq 1$. In particular, given a measure $\mu$ they prove that the infimum is achieved among measures that are singular with respect to $\mu$. Under the additional constraint that the measure has density bounded by 1, they show that the optimal solution is given by the characteristic function of a set.

In \cite{BCL2020} the authors also analyze the perimeter regularization of \eqref{eq:wasserstein}. Namely, they consider the problem
    \begin{equation}\label{eq:wasserstein-perimeter}
    \inf \Big\{ P(E) + \lambda W_p({\mathcal{L}^N \res} E,{\mathcal{L}^N \res}F) \colon  E, F \subset \R^N, \,|E\cap F|=0, \, |E|=|F|=1 \Big\},
\end{equation}
and show that minimizers exist for arbitrary $\lambda>0$, if the admissible sets $E$ and $F$ are required to be subsets of a bounded domain $\Omega$. This problem (with $p=1$) is proposed by Peletier and R\"{o}ger as a simplified model for lipid bilayer membranes where the sets $E$ and $F$ represent the densities of the hydrophobic tails and hydrophilic heads of the two part lipid molecules, respectively \cite{PR2009,LPR2014}. The perimeter term represents the interfacial energy arising from hydrophobic effects, while the Wasserstein term models the weak bonding between the heads and tails of the molecules.

When posed over the unbounded space, Buttazzo, Carlier and Laborde prove the existence of minimizers for the problem \eqref{eq:wasserstein-perimeter} in two dimensions. Xia and Zhou \cite{XZ2021} extend this result to higher dimensions but under the additional assumptions that $\lambda$ is sufficiently small and that $p<n/(n-2)$. Recently, Novack, Venkatraman and the third author \cite{NTV2023} prove that minimizers to \eqref{eq:wasserstein-perimeter} exist in any dimension \emph{and} for all values of $\lambda>0$ and $p \in [1,\infty)$. Simultaneously, Candau-Tilh and Goldman \cite{C-TG2022} also obtain the existence of minimizers via an alternative argument and characterize global minimizers in the small $\lambda$ regime. The analysis in \cite{C-TG2022} and \cite{NTV2023} show that there is a direct competition between the perimeter and the Wasserstein terms in \eqref{eq:wasserstein-perimeter}. This, also as pointed out by Rupert Frank to the third author, leads to the question whether the functional \eqref{eq:wasserstein} is \emph{maximized} when the set $E$ is a ball. Here, we resolve this question for $p>1$.

It often happens that we need to relax a functional to exploit some compactness. We denote by $\AA_m$ the class of admissible densities with mass $m$ that we use to relax the problem, i.e., 
\[
    \AA_m\coloneqq \left\{\rho\in L^1(\R^N)\colon0\leq \rho\leq 1, \int\rho \, dx=m\right\}.
\]
We will use the shorthand notation $\AA\coloneqq\AA_1$ when we deal with probability densities. We define the relaxation of \eqref{eq:wasserstein} to densities $\rho$ with $0\leq \rho\leq 1$ as follows:
\begin{equation}\label{eq:wasserstein-densities}
\begin{split}
    \W_p(\rho)\coloneqq \inf\left\{W_p(\rho,\rho')\colon 0 \leq \rho', \ 0\leq \rho+\rho'\leq 1, \int\rho' \, dx=\int\rho \, dx\right\}.
\end{split}
\end{equation}

\smallskip

Our main result is the following theorem.

\smallskip
\begin{mainthm}
    The unique maximizer of \eqref{eq:wasserstein-densities} in the class $\AA_m$, up to translations, is the characteristic function of a ball $B$ with $|B|=m$.
\end{mainthm}

\medskip

By \cite[Proposition 5.2]{DMSV2016} in the case $p=2$, and by the same result combined with \cite[Theorem 3.10]{BCL2020} and  \cite[Proposition 2.1]{C-TG2022} in the case $p\neq 2$, the expression \eqref{eq:wasserstein-densities} extends the definition on sets given in \eqref{eq:wasserstein}. By these results, we also have that for any $\rho\in\AA_m$ there is a unique density $\eta_{\rho}$ realizing \eqref{eq:wasserstein-densities} when $p>1$. Note that, for $p>1$ \cite[Theorem 2.44]{V2003} guarantees that there is only one optimal transport plan $\pi_{\rho}$ between $\rho$ and $\eta_{\rho}$, and it is induced by a map.

The class of transport plans, which we will call \textit{admissible plans}, that play a role in the definition of $\W_p(\rho)$ is given by
\begin{equation*}
\begin{split}
    \AP_{\rho} \coloneqq \left\{\pi\in\M_+(\R^N\times \R^N)\colon (p_1)_{\#}\pi = \rho\L^N, \ (p_2)_{\#}\pi\leq (1-\rho)\L^N\right\},
\end{split}
\end{equation*}
where $\M(\R^N)$ denotes the set of signed Borel measures in $\R^N$, and $\M_+(\R^N)\subset \M(\R^N)$ denotes the set of non-negative measures. Here $p_1$ and $p_2$ are the two usual projections from $\R^N\times \R^N$ in $\R^N$. Notice that, thanks to the properties of the push-forward, it is automatically true that the density of $(p_2)_{\#}\pi$ with respect to $\L^N$ belongs to $\AA_m$ whenever $\rho\in\AA_m$ and $\pi\in \AP_{\rho}$.

\begin{remark}\label{rem:metric-measure-setting}
    We point out that the energy $\W_p(\rho)$ can be defined on any metric space with a reference measure (in our case, the euclidean space $\R^N$ endowed with $\L^N$). If $(X,d)$ is a Polish metric space, and $\gamma\in\M_+(X)$ is a Borel measure, then for any density $\rho\colon X\to[0,1]$ we can define its Wasserstein energy as
    \[
        \W_p(\rho) \coloneqq \inf \left\{W_p(\rho\gamma,\rho'\gamma)\colon 0 \leq \rho', \ \rho+\rho'\leq1, \ \int\rho'\,d\gamma = \int\rho \,d\gamma\right\},
    \]
    and the $p$-Wasserstein distance can be defined in any metric space. We continue to denote by $\AP_{\rho}$ the set of admissible plans, i.e.
    \[
        \AP_{\rho} = \left\{\pi\in\M_+(X\times X)\colon (p_1)_{\#}\pi = \rho\gamma, \ (p_2)_{\#}\pi\leq(1-\rho)\gamma\right\}.
    \]    
    We cannot expect to have many invariance properties in an abstract setting, but some analytic-flavoured features could be retrieved in wide generality. We will not use this abstract formulation in this paper, with the exception of Proposition~\ref{prop:maximizer-1D} where we consider the space $X=\R^+$ with a weight. This appears because in Section~\ref{sec:maximizer} we reduce to radial densities, and it is convenient to look at them as $1$-dimensional densities (a weight pops up because of the coarea formula).
\end{remark}

\medskip

\subsection*{Plan of the paper}
In Section~\ref{sec:preliminary results} we introduce some preliminary results that are useful for the problem. After recalling briefly some well-known theorems about the existence and uniqueness of the optimal transport map, we introduce some very simple properties of the functional $\W_p$ that were essentially already present in the literature for slightly different problems. In particular, Lemma~\ref{lem:full-ball} is devoted to the saturation of the constraint in a certain region, and Corollary~\ref{cor:bounded-transport-distance} provides a uniform control on the transport distance. These two results are quite robust, as they do not require any geometric property of the Euclidean space, but just its metric-measure structure. Lemma~\ref{lem:weak-continuity-wasserstein} and Lemma~\ref{lemma:symmetry-transport} are an original contribution. The first one, which shows the continuity of the functional $\W_p$ with respect to the weak$*$ convergence (when there is no loss of mass), is fundamental for the existence of maximizers for $\W_p$. The second one, on the other hand, shows that some symmetries of a density $\rho$ can be inherited by the optimal plan $\pi_{\rho}$ that realizes $\W_p(\rho)$.
In Section~\ref{sec:maximizer} we deal with the maximizers of $\W_p$, whose existence is proved in Proposition~\ref{thm:existence-maximizers} applying the concentration compactness principle. This is a building block also for our successive characterization of the maximizers, since we combine a symmetrization technique and the uniqueness of the optimal transport plan to show that the maximizers have some symmetry. We proceed as follows:
\begin{enumerate}
    \item prove that the segments maximize a 1-dimensional weighted version of $\W_p$, in Proposition~\ref{prop:maximizer-1D};
    \item prove that, if $\rho$ is a given maximizer, then the optimal transport plan realizing $\W_p(\rho)$ is radial. This is contained in Corollary~\ref{cor:radial-transport-for-maximizers}, as a consequence of Lemma~\ref{lemma:non-crossing};
    \item combine the first two points to show that the maximizers have to be star-shaped sets, and then conclude that the ball is the only possible maximizer thanks to the saturation of the constraint exposed in Lemma~\ref{lem:full-ball}. This is contained in Theorem~\ref{thm:ball-max}, and it is our main contribution.
\end{enumerate}

Finally, in Section~\ref{sec:quant_ineq} we prove a quantitative version of this maximality result in one dimension, where we show that the deficit of maximality is controlled from below by the square of an asymmetry given as the $L^1$ distance between the ball and any density. Our inequality is asymptotically sharp, in the sense that the exponent of the asymmetry cannot be lowered.

\medskip

A few days before submitting this paper, we became aware of the independent work by Candau-Tilh, Goldman and Merlet \cite{C-TGM2023preprint} (posted on arXiv on September 6, 2023) studying the same maximization problem. Their result is more general, as it considers a broader class of cost functions in the transport problem. In particular, they prove that the characteristic function of the ball maximizes \eqref{eq:wasserstein-densities} when the transport cost is of the form $c(x)=h(|x|)$ with a continuous and increasing function $h$ such that $h(0)=0$ and $h\to\infty$ as $|x|\to \infty$. Our strategy, pursued in Section~\ref{sec:maximizer} 
is more geometric, and circumvents the need to introduce Kantorovich potentials in the transport problem. While we believe that also our strategy can be extended to cover more general cost functions, our proofs rely on the metric structure induced by the $p$-Wasserstein distance as well as on the homogeneity of the cost function which allows us to use scaling properties of the energy.

\medskip

\subsection*{Notation}
Throughout the paper, with an abuse of notation, we will denote the Wasserstein distance between two disjoint set, $W_p(\L^N\res E,\L^N\res F)$, by $W_p(E,F)$. By $B_r(x)$ we will denote the open ball of center $x$ and radius $r$, and we will write $B_r$ for $B_r(0)$. The cube of side length $2l$ centered at the origin will be denoted by $Q_l = [-l,l]^N\subset \R^N$; hence, $Q_l(x) = x+Q_l$. We will use the notation $E_k\weakstar f$ to denote the convergence of a sequence of sets $\{E_k\}_{k\in\N}$ in the sense that the sequence of measures $\{\L^N\res E_k\}_{k\in\N}$ weak$*$ converges to the measure $f\L^N$.

For $\rho\in\AA$ by $\eta_{\rho}$ we will denote any density in $\AA$ such that $\W_p(\rho) = W_p(\rho,\eta_{\rho})$. Note that for $p>1$ we have that $\eta_\rho$ is unique (cfr. \cite[Remark 3.11]{BCL2020}). Similarly, for $\rho\in\AA$, $\pi_{\rho}$ will denote the optimal plan $\W_p^p(\rho) = \int|x-y|^p\,d\pi_{\rho}(x,y)$, and $T_{\rho}$ is the optimal transport map that induces $\pi_{\rho}$. If we have a density $f$, we will sometimes use the short-hand notation $T_{\#}f$ to denote the push forward of the measure $T_{\#}(f\L^N)$.

\medskip
\section{Preliminary results}\label{sec:preliminary results}

\subsection{The optimal transport problem} We introduce in this section the optimal transport problem The general theory is well developed, and goes far beyond the needs of this paper. We state the relevant results just in the setting that we need. The interested reader may find much more general statements, and much deeper developments, in the references that we cite, as well as in other books on the subject. A crucial restriction that we impose is to work with cost $c(x)=|x|^p$ with $p>1$ and (mostly) in the Euclidean space $\R^N$. This plays a role when we characterize the maximizers of $\W_p$ since we use some uniqueness result valid for these special cost functions. Other parts of our strategy work also for $p=1$ with a slightly different discussion. The next definitions describe rigorously our framework. 

A general setting for the optimal transport problem is that of Polish metric spaces, which are defined as follows.
\begin{defin}[Polish metric space]
    A metric space $(X,d)$ is \textit{Polish} if it is complete and separable.
\end{defin}
\begin{defin}[Push forward]
    Let $(X,d_X)$ and $(Y,d_Y)$ be two Polish metric spaces. Given $f:X\to Y$ a Borel function, and given a measure $\mu\in\M(X)$, the \textit{push forward} of $\mu$ induced by $f$ is a new measure denoted by $f_{\#}\mu$. It is defined as follows: for every $A\subset Y$ Borel, we have that
    \[
        (f_{\#}\mu)(A) = \mu(f^{-1}(A)).
    \]
\end{defin}

Given a Polish metric space $(X,d)$, a real exponent $p>1$, and two measures $\mu_1,\mu_2\in \M_+(X)$ with $\mu_1(X) = \mu_2(X)<+\infty$, we can consider the optimal transport problem with cost $c(x)=|x|^p$:
\[
    W_p^p(\mu_1,\mu_2) = \inf\left\{\iint_{X\times X} |x-y|^p\, d\pi(x,y): \pi\in\M_+(X\times X): (p_1)_{\#}\pi = \mu_1, (p_2)_{\#}\pi = \mu_2\right\}.
\]
It is well known that for every couple of marginals $\mu_1$ and $\mu_2$ the infimum is attained (see \cite[Theorem 1.3]{V2003} for a more general result). In some special cases, there are some structure theorems for the optimal transport plans, i.e. those measures $\pi$ that realize the aforementioned infimum. The following is such a result that holds for strictly convex costs.

\begin{theorem}{\cite[Theorem 2.44]{V2003}}\label{thm:existence-uniqueness-transport}
    Let $p>1$ be given, and $\mu_1,\mu_2\in\M_+(\R^N)$ be two measures with $\mu_1(\R^N) = \mu_2(\R^N)<+\infty$. Suppose that $\mu_1\ll \L^N$ and that $W_p(\mu_1,\mu_2)<+\infty$. Then, there is a unique optimal transport plan $\pi$, and it is of the form
    \[
        \pi = (\Id,T)_{\#}\mu_1,
    \]
where $T$ denotes the unique optimal transport map.
\end{theorem}

In Section~\ref{sec:maximizer} it is crucial to characterize the maximizers in one dimension to later pass to higher dimension. Our task is simplified in one dimension because the transport problem has a very easy solution.

\begin{theorem}{\cite[Remarks 2.19]{V2003}}\label{thm:transport-1D}
    Let $p>1$ be given, and let $\mu_1,\mu_2\in\M_+(\R)$ be two measures with $\mu_1(\R) = \mu_2(\R) <+\infty$. If they are non-atomic, then the only optimal transport map realizing $W_p(\mu_1,\mu_2)$ is monotone.
\end{theorem}

\medskip

\subsection{Properties of $\W_p$}
The most basic fact is the following existence theorem.
\begin{theorem}{\cite[Section 5]{DMSV2016}}\label{thm:existence-W_p}
   Let $p>1$ be given. For any $m>0$ and for any $\rho\in\AA_m$, there exists a unique density, called $\eta_{\rho}\in\AA_m$, realizing the infimum in \eqref{eq:wasserstein-densities}.
\end{theorem}
Combining this result with Theorem~\ref{thm:existence-uniqueness-transport} we obtain the existence and uniqueness of the optimal transport plan $\pi_{\rho}$ and the map inducing it, called $T_{\rho}$, which satisfy
\[
    \W_p^p(\rho) = W_p^p(\rho,\eta_{\rho}) = \int |x-y|^p \, d\pi_{\rho}(x,y) = \int |x-T_{\rho}(x)|^p\rho(x) \, dx.
\]
We point out that the objects $\eta_{\rho}$, $\pi_{\rho}$ and $T_{\rho}$ all depend implicitly on $p$. We do not stress that dependence because we suppose $p>1$ to be fixed in the whole paper.

The following lemma establishes a key a geometric property of the optimal plan $\pi_{\rho}$. In the case of the quadratic cost ($p=2$) on $\mathbb{R}^N$, this property is known, see for example \cite[Corollary 2.4]{CMcC2010} and  \cite[Lemma 5.1]{DMSV2016}. The proof of the following lemma is purely metric and uses only the optimality of $\eta_{\rho}$.

\begin{lemma}[Interior ball condition] \label{lem:full-ball}
    Let $(X,d)$ be a Polish metric space, and let $\gamma\in\M_+(X)$ be a given measure. Let $\rho\colon X\to[0,1]$ be a Borel density. If $\pi$ is an optimal plan to compute $\W_p(\rho)$ and $(x,y)\in\spt \pi$, then
    \begin{equation}\label{eq:full-ball}
        (p_2)_{\#}\pi = (1-\rho)\gamma \qquad \gamma-\text{a.e. in }B_{|y-x|}(x).
    \end{equation}
    Moreover, 
    $ (p_2)_{\#}\pi \geq \min\{1-\rho,\rho\}\gamma$.
\end{lemma}

\begin{proof}
    We first show that $ (p_2)_{\#}\pi$ saturates the constraint in the ball, and the second statement will follow easily. The idea is very simple: if $\pi$ does not saturate the constraint in that ball, then we can lower the energy of $\rho$ adding some mass close to $x$. We define $r=|y-x|$. Let us suppose by contradiction that there exist $\eps,\delta>0$ and a set $E\subset B_{r-4\delta}(x)$ with $\gamma(E)$ strictly positive and finite and such that
    \[
        (1-\rho)\gamma-(p_2)_{\#}\pi \geq \eps \gamma\qquad \text{in }E.
    \]
    We take $\mu_1 = (p_1)_{\#}(\pi\res B_{\delta}(x)\times B_{\delta}(y))$ and $\mu_2 = \eps\gamma\res E$, and we modify $\pi$ in the following way: we take $0<t<\min\{1,\mu_1(X)/\mu_2(X)\}$, and we take
    \begin{equation*}
    \begin{split}
        \tilde \pi = \pi - t\frac{\mu_2(X)}{\mu_1(X)}\pi\res(B_{\delta}(x)\times B_{\delta}(y)) + \frac{t}{\mu_1(X)}\mu_1\times\mu_2.
    \end{split}
    \end{equation*}
    One can check that $\tilde \pi\in\AP_{\rho}$ thanks to our choice of $t$. Since $\pi$ is an optimal plan to compute $\W_p(\rho)$, we have that
    \begin{equation*}
    \begin{split}
        0&\leq \int  |x'-y'|^p\,(d\tilde\pi-d\pi) \\
        &\leq= -t\frac{\mu_2(X)}{\mu_1(X)}\int_{B_{\delta}(x)\times B_{\delta}(y)} |x'-y'|^p\,d\pi+\frac{t}{\mu_1(X)}\int |x'-y'|^p\,d\mu_1 \, d\mu_2\\
        &\leq -t\frac{\mu_2(X)}{\mu_1(X)}(r-2\delta)^p\mu_1(X)+\frac{t}{\mu_1(X)}(r-4\delta+\delta)^p\mu_1(X)\mu_2(X)\\
        &= t\mu_2(X)\left[(r-3\delta)^p-(r-2\delta)^p\right]<0,
    \end{split}
    \end{equation*}
    and thus we reach a contradiction.

    We now address the second inequality. Suppose by contradiction that the opposite inequality holds in a set $E\subset X$ with $\int_E \rho d \gamma>0$. Then, thanks to what we have proved so far, we know that the set
    \begin{equation}\label{eq:no-motion}
        \{x\in E\colon \spt \pi\cap (\{x\}\times  X) = (x,x)\}
    \end{equation}
    has full $\gamma$-measure in $E$. In fact, if this was not the case, then we could find $E'\subset E$ with $\gamma(E')>0$ and such that, for every $x\in E'$, there exists $y\in X\setminus \{x\}$ such that $(x,y)\in\spt \pi$. Then, using \eqref{eq:full-ball} we find an open covering of $E'$ where the contradiction hypothesis is not satisfied, contrary to the definition of $E$. Condition \eqref{eq:no-motion} means that we are not moving mass in $E$, and thus
    \[
        (p_2)_{\#}(\pi\res (E\times X)) = (p_1)_{\#}(\pi\res(E\times X)) = \rho\Chi{E}\gamma.
    \]
    But then $ (p_2)_{\#}\pi \geq (p_2)_{\#}(\pi\res (E\times X)) = \rho\Chi{E}\gamma$, which is incompatible with our contradiction hypothesis. 
\end{proof}

\begin{corol}\label{cor:bounded-transport-distance}
    Consider the functional $\W_p$ on the Euclidean space $\R^N$ with the usual metric and the Lebesgue measure $\L^N$. There exists a constant $C_N<+\infty$ such that, for any $\rho\in\AA_m$ and for any $(x,y)\in \spt \pi_{\rho}$, we have that
    \begin{equation}\label{eq:bounded-transport-distance}
        |x-y|\leq C_Nm^{\frac{1}{N}}.
    \end{equation}
    Here, $\pi_\rho$ is any optimal transport plan $\pi_{\rho}$ associated to $\rho$ and $\eta_{\rho}$. In particular $\W_p^p(\rho)\leq C_Nm^{1+\frac{p}{N}}$.
\end{corol}

\begin{proof}
This is a consequence of Lemma~\ref{lem:full-ball}. If $r>0$ we have that
\[
    \int_{B_r(z)}\rho+\rho'\, dx\leq 2m
\]
for all $\rho'\in\AA_m$ and all $z\in\R^N$.
Thus, if we fix $r$ such that $|B_r|=2m$,  
then the conclusion \eqref{eq:full-ball} of Lemma~\ref{lem:full-ball} fails for any pair of points $(x,y)$ with $|x-y|>r$. Hence such a pair cannot lie in $\spt\pi_\rho$. It follows that every pair $(x,y)\in\spt\pi_{\rho}$ satisfies
\[
    |x-y|\leq r = \left(\frac{2}{\omega_N}\right)^{\frac1N}m^{\frac1N}.
\]
The estimate on $\W_p(\rho)$ follows by integrating this inequality with respect to the measure $\pi_{\rho}$.
\end{proof}

\begin{remark}\label{rem:scaling-wasserstein}
    We report here the scaling behavior of the energy $\W_p$, which is established in \cite[Lemma 2.5]{NTV2023} for sets. Let $\rho$ be a density satisfying the constraint $0\leq \rho\leq 1$ and let $t>0$ be a given constant. If we rescale $\tilde \rho(x)=\rho(x/t)$, then  $\W_p^p(\tilde \rho) = t^{p+N}\W_p^p(\rho)$. In fact, it is sufficient to consider the density $\eta_{\rho}(\cdot/t)$, rescaling appropriately the transport map.
\end{remark}

\begin{lemma}[Continuity of $\W_p$]\label{lem:weak-continuity-wasserstein}
    Let $\rho\in\AA_m$ be a given density and let $\{\rho_n\}_{n\in\N}\subset \AA_m$ be a sequence such that $\rho_n\weakstar\rho$. Then, the limit of $\W_p(\rho_n)$ exists and $\W_p(\rho) = \lim_n \W_p(\rho_n)$.
\end{lemma}

\begin{proof} We prove this proposition in two steps. In the first step we establish that for any $p \geq 1$ \eqref{eq:wasserstein-densities} is the lower semicontinuous envelope of the functional in \eqref{eq:wasserstein} in the class $\AA_m$ with respect to the weak-$*$ topology. As a consequence, $\W_p$ is lower semicontinuous in $\AA_m$. In the second step we obtain the upper semicontinuity of $\W_p$ in $\AA_m$.

\medskip

\noindent\emph{Step 1.}    Thanks to Remark~\ref{rem:scaling-wasserstein} it suffices to consider the case $m=1$. Let $\{E_n\}_{n\in\N}$ be a sequence of sets with $|E_n|=1$ such that $E_n\weakstar \rho$ for some $\rho\in\AA$, and let us call $\rho_n=\Chi{E_n}$. Since we preserve the total mass, we know that for any $\eps>0$ there exist $R>0$ and $k\in\N$ such that $\int_{B_R}\rho_n \,dx>1-\eps$ for every $n>k$. Using Corollary~\ref{cor:bounded-transport-distance} we know that the transport distance is uniformly bounded by a constant $C$, and thus $\int_{B_{R+C}}\eta_{\rho_n}\,dx\geq 1-\eps$ for any $n>k$. Therefore, up to a subsequence, we have that also $\eta_{\rho_n}\weakstar \rho'$ for some density $\rho'$ with $\int \rho' \, dx=1$. It is then easy to see that $\rho+\rho'\leq 1$ almost everywhere, and thus
    \[
        \W_p(\rho)\leq W_p(\rho,\rho') \leq \liminf_n W_p(\rho_n,\eta_{\rho_n}) = \W_p(\rho_n),
    \]
    where we used the well-known lower semicontinuity of the Wasserstein distance (it is sufficient to take the weak limit of the optimal transport plans). This proves that the functional in \eqref{eq:wasserstein-densities} is smaller than the lower semicontinuous envelope of $\W_p$ with respect to the weak$*$ topology. Next, we will find a sequence that realizes the equality, proving that our definition of $\W_p(\rho)$ in $\AA$ \emph{is} the lower semicontinuous envelope of the functional defined in \eqref{eq:wasserstein}. 
    
    Given $\rho\in\AA$, for any $n\in\N$ we consider a partition of $\R^N$ with a family of cubes $\mathcal{F}_n = \{Q^k_n\}_{k\in\N}$ with diameter $1/n$. Thanks to the compatibility condition $\rho+\eta_{\rho}\leq 1$, for any $n$ we can find two sets $E_n$ and $F_n$ with $|E_n\cap F_n|=0$ and such that
    \[
        |E_n\cap Q^k_n| = \int_{Q^k_n}\rho \,dx,\qquad |F_n\cap Q^k_n|=\int_{Q^k_n}\eta_{\rho}\,dx,\qquad \forall Q^k_n\in\mathcal{F}_n.
    \]
    It is immediate to see that $E_n\weakstar \rho$ and $F_n\weakstar \eta_{\rho}$ as $n\to+\infty$. Recalling $m=1$, we also note that $W_p(E_n,\rho) \leq \diam (Q^k_n)$ and $W_p(\eta_{\rho},F_n)\leq \diam(Q^k_n)$. To see this, it is sufficient to consider the (non-optimal) transport plan given by
    \begin{equation}\label{eq:plan-lower-semicontinuity}
        \pi_n = \sum_{k\in\N} \frac{1}{|E_n\cap Q^k_n|}(\Chi{E_n\cap Q^k_n}\L^N)\times(\rho\Chi{Q^k_n}\L^N)\in\P(\R^N\times \R^N),
    \end{equation}
    and notice that $|x-y|\leq \diam(Q^k_n)=1/n$ for any $(x,y)\in\spt \pi_n$. The proof of the inequality for $F_n$ and $\eta_{\rho}$ is analogous, and thus we obtain that
    \[
        W_p(E_n,F_n) \leq W_p(E_n,\rho)+W_p(\rho,\eta_{\rho})+W_p(\eta_{\rho},F_n) \leq \frac{2}{n}+W_p(\rho,\eta_{\rho}).
    \]
    This, combined with the first part, shows that
    \[
        \W_p(\rho) = \inf_{E_n\weakstar\rho,|E_n|=m} \liminf_n \W_p(E_n)\qquad \forall \rho\in\AA.
    \]

\medskip

\noindent\emph{Step 2.} We recall that, thanks to Theorem~\ref{thm:existence-uniqueness-transport}, there exists an optimal transport map for every transport problem that we consider in this paper. Up to taking a subsequence, we may suppose that $\lim_n\W_p(\rho_n)$ exists, and argue that $\W_p(\rho) = \lim_n\W_p(\rho_n)$. Since we can extract such a subsequence from any given subsequence of $\{\rho_n\}_n$, this will guarantee the existence of that limit for the entire sequence. 

We proceed by contradiction, and we suppose that there exists $\delta>0$ such that $\W_p(\rho)<\lim_n\W_p(\rho_n)-\delta$. The idea is to modify $\eta_{\rho}$ and produce a competitor to compute $\W_p(\rho_n)$, proving that we cannot have a strict inequality. To proceed with this plan we first truncate the densities to guarantee a convergence in Wasserstein distance. Up to taking another subsequence, we can suppose that $\eta_{\rho_n}\weakstar \rho'$ for some $\rho'\in\AA$ with $\rho+\rho'\leq1$ (using the same argument as in Step 1). Since the sequences $\{\rho_n\}_n$ and $\{\eta_{\rho_n}\}_n$ do not lose mass, for any $\eps<1/2$ there exists $\bar n,k_1\in\N$ such that
    \begin{equation}\label{eq:large-cube}
        \int_{\R^N\setminus Q_{3k_1}}(\rho_n+\eta_{\rho_n})\, dx<\eps\qquad \forall n>\bar n.
    \end{equation}
We will choose $\eps$ later on in order to make some approximations precise enough to obtain a contradiction out of the strict inequality. 

Now take $k_2 = \lceil 3/\eps\rceil$, so that $k_2\eps\in[3,3+\eps]$, and we consider the cube $\bar Q = [-k_1k_2\eps,k_1k_2\eps]^N$. It is easy to see that we can partition $\R^N$ with a family $\mathcal{F} = \{Q^k\}_{k\in\N}$ of cubes with side length equal to $\eps$ and such that $|Q^k\cap \bar Q|\in\{0,\eps^N\}$ (i.e. $\mathcal{F}$ contains two disjoint subfamilies that partition $\bar Q$ and $\R^N\setminus\bar Q$). Moreover, it is also possible to find a partition of $\R^N\setminus\bar Q$ with a family $\tilde{\mathcal{F}}=\{\tilde Q^k\}_{k\in\N}$ of cubes with side length $k_2\eps$. We will use the first partition to control the cost of an approximation of $\eta_{\rho}$ inside $\bar Q$, where we move mass at short distance. The second one, on the other hand, will be used to estimate the energy carried by the mass outside of that cube (thanks to \eqref{eq:large-cube}, that mass is small). We call $T$ the optimal transport map between $\rho$ and $\eta_{\rho}$, and for any $n$ we define the truncated densities $\tilde\rho_n = \rho_n\Chi{\bar Q}$. For any $n$ we also take $L_n>0$ such that $\int_{Q_{L_n}}\rho \,dx = \int\tilde\rho_n\,dx $, and we define the densities $\zeta_n\coloneqq \rho\Chi{Q_{L_n}}$ and $\zeta'_n\coloneqq (T_{\rho})_{\#}\zeta_n$. Since $\rho_n\weakstar\rho$, then $\tilde\rho_n\weakstar \rho\Chi{\bar Q}$ and we can choose the sequence $\{L_n\}_n$ to be bounded. Moreover, we have that $\zeta_n\weakstar \rho\Chi{\bar Q}$. Since the supports of the truncated densities are equibounded, then the $p$th-moment of $\zeta_n$ converges, as well as the $p$th-moment of $\tilde\rho_n$, and thus $W_p(\tilde\rho_n,\zeta_n)\to0$ (see e.g. \cite[Theorem 7.12]{V2003})

We take $h^1_n$ any non-negative density such that $\rho_n+h_n^1\leq 1$ and for any $k\in\N$
\[
    \int_{Q^k}h^1_n \,dx = \min\left\{\int_{Q^k}\zeta'_n \,dx,\ \int_{Q^k}1-\rho_n\,dx\right\}.
\]
Since $\zeta_n' = (T_{\rho})_{\#}\zeta_n$, we can apply Corollary~\ref{cor:bounded-transport-distance} and see that $\spt h^1_n$ is contained in $Q_{L_n+C}$ for any $n$, where $C$ is a constant depending only on $N$. Since $\tilde\rho_n\weakstar \rho\Chi{\bar Q}$ and $\zeta'_n\weakstar (T_{\rho})_{\#}(\rho\Chi{\bar Q})$, then we have that $\norm{h^1_n}_1-\norm{\zeta'_n}_1\to0$ (notice that here only a finite number of cubes in $\mathcal{F}$ play an active role). We choose any non-negative density $h^2_n$ with $\spt h^2_n\subset 3\bar Q$ and such that
    \[
        \rho_n+h^1_n+h^2_n\leq 1\qquad \text{and}\qquad\norm{h^1_n+h^2_n}_1 = \norm{\tilde\rho_n}_1,
    \]
    and our candidate to compute $\W_p(\tilde\rho_n)$ will be $\tilde\rho'_n\coloneqq h^1_n+h^2_n$. Observe that, by definition of $h^1_n$ and thanks to the properties of the push-forward of measures, we have that $\norm{h^1_n}_1\leq \norm{\zeta'_n}_1 = \norm{\zeta_n}_1 = \norm{\tilde\rho_n}_1$.
    Thanks to the triangle inequality for the $p$-Wasserstein distance, we have that
    \[
        W_p(\tilde \rho_n,\tilde\rho'_n)\leq W_p(\tilde\rho_n,\zeta_n)+W_p(\zeta_n,\zeta_n')+W_p(\zeta_n',\tilde\rho'_n).
    \]
    The first term on the right hand side is going to $0$ because, as we already noticed, the sets $\spt\tilde\rho_n$ and $\spt\zeta_n$ are uniformly bounded and these densities are converging to $\rho\Chi{\bar Q}$. Hence, up to taking $\bar n$ large enough, we can suppose that $W_p(\tilde\rho_n,\zeta_n)<\eps$. Likewise, the last term is controlled by $\eps$, and we use a plan similar to \eqref{eq:plan-lower-semicontinuity} to show this.
    
    We choose a density $\zeta''_n\leq \zeta'_n$ such that
    \[
        \int_{Q^k}\zeta''_n \,dx= \int_{Q^k}h^1_n \,dx \qquad \forall k\in\N,
    \]
    and we consider the plan
    \[
        \tilde\pi_n = \sum_{k\in\N}\frac{1}{\norm{h^1_n\Chi{Q^k}}_1}(\zeta''_n\Chi{Q^k}\L^N)\times(h^1_n\Chi{Q^k}\L^N)+\frac{1}{\norm{h^2_n}_1}((\zeta'_n-\zeta''_n)\L^N)\times (h^2_n\L^N),
    \]
    where the sum is intended to run only on the indices for which $h^1_n\Chi{Q^k}$ is not identically zero. 
    Using $\tilde \pi_n$ as test plan to compute $W_p(\zeta'_n,\tilde\rho'_n)$ we obtain the following upper bound:
    \begin{equation*}
        W_p^p(\zeta'_n,\tilde\rho'_n)\leq \int |x-y|^p\,d\tilde\pi_n(x,y) \leq C\eps^p+\diam(\spt h^2_n+\zeta'_n)\left(\norm{\zeta'_n}_1-\norm{h^1_n}_1\right)\leq C\eps^p,
    \end{equation*}
    where we used that the mass of $h^1_n$ remains inside the small cubes with side length $\eps$, and the remaining mass is transported at finite distance in any case (the constant $C$ depends only on $N$ and $p$). The last inequality holds if we take $\bar n$, and thus $n$, large enough, and if we adjust the constant $C$. Adding up the various terms, we conclude that for any $n>\bar n$ there is an optimal transport plan $\pi_n$ for $\tilde\rho_n$ and $\tilde\rho'_n$ such that
    \begin{equation*}
    \begin{split}
        W_p(\tilde\rho_n,\tilde\rho'_n) = \left(\int|x-y|^p\,d\pi_n(x,y)\right)^{\frac{1}{p}}\leq W_p(\zeta_n,\zeta'_n)+C\eps.
    \end{split}
    \end{equation*}
    To conclude, we observe that the cubes in $\tilde{\mathcal{F}}$ are so large that we can find a non-negative density $h^3_n$ such that $\rho_n+\tilde\rho'_n+h^3_n\leq 1$ and
    \[
        \int_{\tilde Q^k}h^3_n \,dx = \int_{\tilde Q^k}\rho_n\,dx \qquad \forall k\in\N.
    \]
    Therefore, we consider the plan $\gamma_n$ associated to $\rho_n$ and $\tilde\rho'_n+h^3_n$ defined as
    \[
        \gamma_n = \pi_n + \sum_{k\in\N}\frac{1}{\norm{\rho_n\Chi{\tilde Q^k}}_1}(\rho_n\Chi{\tilde Q^k}\L^N)\times(h^3_n\Chi{\tilde Q^k}\L^N),
    \]
    again summing only on the cubes with non-trivial measure. This gives the following estimate for $W_p^p(\rho_n,\tilde\rho'_n+h^3_n)$:
    \begin{equation*}
    \begin{split}
        W_p^p(\rho_n,\tilde\rho'_n+h^3_n)&\leq \Big(W_p(\zeta_n,\zeta'_n)+C\eps\Big)^p+C\norm{h^3_n}_1\\
        &\leq \Big(\W_p(\rho)+C\eps\Big)^p+C\norm{h^3_n}_1\\
        &\leq \Big(\W_p(\rho_n)-\delta+C\eps\Big)^p+C\eps.
    \end{split}
    \end{equation*}
    Since $\delta>0$ is fixed and since the constant $C$ in that estimate depends only on $N$ and $p$, we can find $\eps$ small enough so that $W_p^p(\rho_n,\tilde\rho'_n+h^3_n)<\W_p^p(\rho_n)$, and this is impossible since $\tilde\rho'_n+h^3_n$ is a competitor in the definition of $\W_p(\rho_n)$.
\end{proof}

The next lemma describes particular symmetries of the problem \eqref{eq:wasserstein-densities} which are crucial in proving properties of maximizers of $\W_p$ in the next section.

\begin{lemma}[Symmetries of the transport problem]\label{lemma:symmetry-transport}
    Let $F\colon\R^N\to\R^N$ be an isometry and let $\rho\in\AA$ be a given density such that $F_{\#}(\rho\L^N) = \rho\L^N$. Then the following hold:
    \begin{enumerate}
        \item $F_{\#}(\eta_{\rho}\L^N) = \eta_{\rho}\L^N$ and $\tilde F_{\#}\pi_{\rho} = \pi_{\rho}$, where $\tilde F$ is the map from $\R^N\times \R^N$ into itself defined as $\tilde F(x,y) = (F(x),F(y))$.
        \item If $F$ is a reflection of the form $F(x) = x-2\scal{x}{\nu}\nu$ for some $\nu\in\S^{N-1}$, then we have that
        \begin{equation}\label{eq:split-hyperplane}
            \pi_{\rho}\left(\{(x,y)\colon\scal{x}{\nu}\scal{y}{\nu}<0\}\right) = 0.
        \end{equation}
        In other words, $\pi_{\rho}$ does not transport mass from one side of the reflection hyperplane $\{x\colon\scal{x}{\nu}=0\}$ to the other. 
    \end{enumerate}
\end{lemma}

\begin{proof}
    We recall that the optimal plan $\pi_{\rho}$ is unique (see Theorem~\ref{thm:existence-uniqueness-transport}). Also, notice that $F_{\#}(\rho\L^N)$ and $F_{\#}(\eta_{\rho}\L^N)$ are absolutely continuous with respect to the Lebesgue measure, and we have that $F_{\#}(\rho\L^N) = (\rho\circ F)\L^N$ and $F_{\#}(\eta_{\rho}\L^N) = (\eta_{\rho}\circ F)\L^N$. Therefore, it is trivial to see that $F_{\#}(\rho\L^N)\in\AA$, $F_{\#}(\eta_{\rho}\L^N)\in\AA$ and $F_{\#}((\rho+\eta_{\rho})\L^N)\leq \L^N$.
    
    It is easy to see that $\tilde \pi_{\rho} = (\tilde F)_{\#}\pi_{\rho}$ is a transport plan associated to $F_{\#}(\rho\L^N)$ and $F_{\#}(\eta_{\rho}\L^N)$: by the properties of the push forward, we have that $(p_1\circ \tilde F)_{\#}\pi_{\rho} = (p_1)_{\#}(\tilde F_{\#}\pi_{\rho})$, and $p_1\circ \tilde F = F\circ p_1$, therefore $(p_1)_{\#}\tilde \pi_{\rho} = F_{\#}(\rho\L^N)$. An analogous property holds for the second projection $p_2$, and thus $\tilde\pi_{\rho}$ has the correct marginals. Then, we consider the plan $(\pi_{\rho}+\tilde \pi_{\rho})/2$, whose marginals are $\rho\L^N$ and $\frac{1}{2}(\eta_{\rho}+\eta_{\rho}\circ F)\L^N$, and we observe that
    \begin{equation*}
    \begin{split}
        \W_p^p(\rho) &\leq \frac{1}{2}\int |x-y|^p\,d\pi_{\rho}(x,y)+\frac{1}{2}\int|x-y|^p \,d\tilde F_{\#}\pi_{\rho}(x,y)\\
        &= \frac{1}{2}\int |x-y|^p\,d\pi_{\rho}(x,y)+\frac{1}{2}\int|F(x)-F(y)|^p \,d\pi_{\rho}(x,y) = W_p^p(\rho,\eta_{\rho}),
    \end{split}
    \end{equation*}
    where we used that $F$ is an isometry to obtain the last identity. This implies that $\eta_{\rho}\circ F$ is also an optimal density to compute $\W_p(\rho)$. Since there exists a unique density which realizes $\W_p(\rho)$, then $\eta_{\rho}\L^N = F_{\#}(\eta_{\rho}\L^N)$ and $\tilde F_{\#}\pi_{\rho}=\pi_{\rho}$.

    In order to prove (ii), suppose that $F(x)=x-2\scal{x}{\nu}\nu$ for some $\nu\in\S^{N-1}$. 
    From the previous point we know that $\pi_{\rho}$ satisfies $\tilde F_{\#}\pi_{\rho} = \pi_{\rho}$. We want to prove that, whenever \eqref{eq:split-hyperplane} does not hold, we can find a better plan, contradicting the definition of $\pi_{\rho}$. In fact, we consider the plan
    \[
        \tilde \pi_{\rho} = \pi_{\rho}\res (H_1\times H_1)+\pi_{\rho}\res(H_2\times H_2)+(\Id,F)_{\#}(\pi_{\rho}\res (H_1\times H_2))+(\Id,F)_{\#}(\pi_{\rho}\res (H_2\times H_1)),
    \]
    where $H_1=\{x\colon\scal{x}{\nu}>0\}$ and $H_2=F(H_1)=\{x\colon\scal{x}{\nu}<0\}$. We observe that, since $(p_1)_{\#}\pi_{\rho}$ and $(p_2)_{\#}\pi_{\rho}$ are absolutely continuous with respect to Lebesgue measure, then $\pi_{\rho}$ does not give mass to $\bdry (H_i\times H_j)$ for any $i,j\in\{1,2\}$. Therefore, $\tilde \pi_{\rho}$ is a probability measure, and the well-known properties of the push-forward operation guarantee that $(p_1)_{\#}\tilde\pi_{\rho} = \rho\L^N$. 
    Since $\pi_{\rho}=\tilde F_{\#}\pi_{\rho}$ and $\tilde F(H_1\times H_2) = H_2\times H_1$, then $\pi_{\rho}\res (H_1\times H_2) = \tilde F_{\#}(\pi_{\rho}\res(H_2\times H_1))$. With this observation we arrive to
    \begin{align*}
        \left((p_2)_{\#}\tilde\pi_{\rho}\right)\res H_1 &= (p_2)_{\#}\left(\pi_{\rho}\res(H_1\times H_1)\right)+(p_2)_{\#}\left((\Id,F)_{\#}(\pi_{\rho}\res(H_1\times H_2))\right)\\
        &=(p_2)_{\#}\left(\pi_{\rho}\res(H_1\times H_1)\right)+(p_2)_{\#}\left((\Id,F)_{\#}\tilde F_{\#}(\pi_{\rho}\res(H_2\times H_1))\right)\\
        &=(p_2)_{\#}\left(\pi_{\rho}\res(H_1\times H_1)\right)+(p_2)_{\#}\left((F,\Id)_{\#}(\pi_{\rho}\res(H_2\times H_1))\right)\\
        &=(p_2)_{\#}\left(\pi_{\rho}\res(H_1\times H_1)\right)+(p_2)_{\#}\left(\pi_{\rho}\res(H_2\times H_1)\right)\\
        &=(p_2)_{\#}(\pi_{\rho}\res (\R^N\times H_1)) = ((p_2)_{\#}\pi_{\rho})\res H_1,
    \end{align*}
    where we used that $(\Id,F)\circ \tilde F = (F,\Id)$ and the fact that $F$ is an isometry to pass from the second to the third line. Arguing in the same way, one can also see that $((p_2)_{\#}\tilde \pi_{\rho})\res H_2 = ((p_2)_{\#}\pi_{\rho})\res H_2$. This is sufficient to say that $\rho\L^N+(p_2)_{\#}(\tilde \pi_{\rho}\L^N)\leq \L^N$, and thus $\tilde \pi_{\rho}\in\AP_{\rho}$. Now we can compare the costs associated to $\tilde \pi_{\rho}$ and $\pi_{\rho}$. Discarding the common terms, we get that
    \begin{equation}\label{eq:reflection-inequality}
        \int |x-y|^p \, d(\tilde \pi_{\rho}-\pi_{\rho}) = \int_{(H_1\times H_2)\cup (H_2\times H_1)} \left(|x-F(y)|^p-|x-y|^p\right)\,d\pi_{\rho}(x,y),
    \end{equation}
    and a simple geometric argument shows that the function inside the integral is strictly negative. Therefore, if the domain appearing in the right hand side of \eqref{eq:reflection-inequality} has positive $\pi_{\rho}$ measure, then $\tilde \pi_{\rho}$ is a strictly better competitor to compute $\W_p(\rho)$, in contradiction with the definition of $\pi_{\rho}$. To conclude, we observe that we have just proved that $\pi_{\rho}((H_1\times H_2)\cup (H_2\times H_1))=0$, and this is equivalent to \eqref{eq:split-hyperplane}.
\end{proof}

\medskip
\section{Maximizer of $\W_p$}\label{sec:maximizer}

\subsection{Existence of maximizers}
In this section we first prove the existence of maximizers of the energies \eqref{eq:wasserstein-densities} in $\AA$ by applying the concentration compactness principle to a maximizing sequence of densities, where we consider them as measures. Even though we consider a maximization problem, our strategy works since $\W_p$ is continuous with respect to the weak$*$ convergence, as shown in Lemma~\ref{lem:weak-continuity-wasserstein}. Here we state concentration compactness lemma for measures for the convenience of the reader.

\begin{lemma}[Concentration compactness, \cite{S}]\label{lemma:concentration-compactness}
    Let $\mu_n\in\P(\R^N)$ be a given sequence of probability measures. Then there exists a subsequence (not relabelled) such that one of the following holds:
    \begin{enumerate}
        \item \emph{(Compactness)} There exists a sequence of points $x_n\in\R^N$ such that, for every $\eps>0$, there exists $L>0$ large enough such that $\mu_n(Q_L(x_n))>1-\eps$.
        \item \emph{(Vanishing)} For every $\eps>0$ and every $L>0$ there exists $\bar n\in\N$ such that
        \[
            \mu_n(Q_L(x))<\eps\qquad \forall x\in\R^N, \forall n>\bar n.
        \]
        \item \emph{(Dichotomy)} There exist $\lambda \in(0,1)$ and a sequence of points $x_n\in\R^N$ with the following property: for any $\eps>0$, there exists $L>0$ such that, for any $L'>L$  there exist two non-negative measures $\mu^1_n$ and $\mu^2_n$ that satisfy, for every $n$ large enough, the following conditions
        \begin{equation*}
        \begin{split}
            &\mu^1_n+\mu^2_n\leq \mu_n,\\
            &\spt \mu^1_n\subset Q_L(x_n),\quad \spt \mu^2_n\subset \R^N\setminus Q_{L'}(x_n),\\
            &\left|\mu^1_n(\R^N)-\lambda\right|+\left|\mu^2_n(\R^N)-(1-\lambda)\right|<\eps.
        \end{split}
        \end{equation*}
    \end{enumerate}
\end{lemma}

\bigskip

\begin{theorem}\label{thm:existence-maximizers}
    Let $p>1$ be fixed. Then there exists a maximizer of $\W_p$ in $\AA$.
\end{theorem}

\begin{proof}
    Let us consider a maximizing sequence $\rho_n\in\AA$ with $\W_p(\rho_n)\to \sup_{\rho\in\AA}\W_p(\rho)$. Notice that, thanks to Corollary~\ref{cor:bounded-transport-distance}, we have that $\sup_{\rho\in\AA}\W_p(\rho)\leq C<+\infty$ for some constant $C=C(p,N)$. We are going to apply the concentration compactness lemma to $\mu_n=\rho_n\L^N$, and show that the vanishing and dichotomy phenomena do not happen. Then exploiting the invariance of the energy under translations and Lemma~\ref{lem:weak-continuity-wasserstein} we establish the existence of a maximizer.
    
    We first exclude the vanishing case. Up to translations, we can suppose that the points $x_n$ appearing in Lemma~\ref{lemma:concentration-compactness} all coincide with the origin. Suppose by contradiction that, for any $\eps>0$ and any $L>0$ we can find $\bar n\in\N$ such that $\mu_n(Q_L(x))<(\eps/3)^N$ for every $x\in\R^N$. Then, we fix a partition $\mathcal{F}=\{Q^k\}_{k\in\N}$ of $\R^N$ made of cubes with side length $\eps$. Since by hypothesis $\mu_n(Q^k)<|Q^k|/3$ for every $n>\bar n$ and every $k\in\N$, then for every $n>\bar n$ there exists $\rho'_n\in\AA$ such that $\rho_n+\rho'_n\leq 1$ and
    \[
        \int_{Q^k}\rho_n \,dx = \int_{Q^k}\rho'_n\,dx \qquad \forall k\in\N.
    \]
    Using a transport plan similar to $\pi_n$ defined in \eqref{eq:plan-lower-semicontinuity}, it is immediate to see that
    \[
        \W_p^p(\rho_n) \leq W_p^p(\rho_n,\rho'_n) \leq \diam(Q^k)^p = C_{N,p}\eps^p.
    \]
    If we take $\eps$ sufficiently small, we clearly have that $\rho_n$ is not a maximizing sequence for $\W_p$, arriving to a contradiction.

    Now we treat the dichotomy case. Suppose for a contradiction that there exists $\lambda\in(0,1)$ such that, for any $\eps>0$ there exist $\bar n\in\N$, $L>0$ and two sequences of non-negative densities $\rho^1_n$, $\rho^2_n$ that satisfy
    \begin{equation}\label{eq:dichotomy-densities}
    \begin{split}
        &\rho^1_n+\rho^2_n\leq \rho_n\\
        &\spt \rho^1_n\subset Q_L\quad \spt \rho^2_n\subset \R^N\setminus Q_{L+3C_N},\\
        &\left|\int\rho^1_n\,dx -\lambda\right|+\left|\int\rho^2_n\,dx-(1-\lambda)\right|<\eps,
    \end{split}
    \end{equation}
    where $C_N$ is the constant appearing in \eqref{eq:bounded-transport-distance}. 
    
    Since the distance between $\spt \rho^1_n$ and $\spt \rho^2_n$ is larger than $3C_N$, then applying Corollary~\ref{cor:bounded-transport-distance} we obtain that $\W_p^p(\rho^1_n+\rho^2_n) = \W_p^p(\rho^1_n)+\W_p^p(\rho^2_n)$. Combining the first and the third conditions in \eqref{eq:dichotomy-densities}, we get that $\norm{\rho_n-\rho^1_n-\rho^2_n}_1<\eps$, and we define $m^1_n = \norm{\rho^1_n}_1$ and $m^2_n = \norm{\rho^2_n}_1$. Using this fact, and that $\rho^1_n+\rho^2_n+\eta_{\rho^1_n+\rho^2_n}\leq 1$, we deduce that
    \begin{equation}\label{eq:almost-compatible}
        \int (\eta_{\rho^1_n+\rho^2_n}-(1-\rho_n))_+\,dx\leq \eps.
    \end{equation}
    
    We denote by $T_n$ the optimal transport map to compute $\W_p(\rho^1_n+\rho^2_n)$, and we define
    \[
        \zeta_n = \min\{\eta_{\rho^1_n+\rho^2_n},1-\rho_n\},\qquad \tilde\rho_n = (T_n^{-1})_{\#}\zeta_n,
    \]
    so that $\tilde \rho_n$ is an approximation of $\rho^1_n+\rho^2_n$, and it is smaller than that sum. We let $\mathcal{F}=\{Q^k\}_{k\in\N}$ be a partition of $\R^N$ made of cubes with side length equal to $3$, and we can find, as we did before, a density $\zeta'_n$ such that $\rho_n+\zeta_n+\zeta'_n\leq 1$ and
    \[
        \int_{Q^k}\zeta'_n \,dx= \int_{Q^k}\rho_n-\tilde \rho_n \,dx\qquad \forall k\in\N.
    \]
    Therefore, we estimate the energy of $\rho_n$ with the plan
    \[
        \tilde \pi_n = (\Id,T_n)_{\#}\tilde \rho_n + \sum_{k\in\N} \frac{1}{\norm{\zeta'_n\Chi{Q^k}}_1}((\rho_n-\tilde\rho_n)\Chi{Q^k}\L^N)\times (\zeta'_n\Chi{Q^k}\L^N).
    \]
    In fact, combining \eqref{eq:almost-compatible} and the fact that $\norm{\rho_n-\rho^1_n-\rho^2_n}_1\leq \eps$, we have that $\norm{\rho_n-\tilde\rho_n}_1\leq 2\eps$, and thus 
    \begin{equation}\label{eq:dichotomy-estimate}
    \begin{split}
        \W_p^p(\rho_n)&\leq\int|x-y|^p\,d\tilde\pi_n  \leq \int |T_n(x)-x|^p\tilde\rho_n(x)\,dx+2(\diam Q^k)^p \eps\\
        &\leq \W_p^p(\rho^1_n+\rho^2_n)+C_{N,p}\eps\\
        &=\W_p^p(\rho^1_n)+\W_p^p(\rho^2_n)+C_{N,p}\eps\\
        &\leq \sup\left\{\W_p^p(\rho)\colon\rho\in \AA_{m^1_n}\right\}+\sup\left\{\W_p^p(\rho)\colon\rho\in \AA_{m^2_n}\right\}+C_{N,p}\eps.
    \end{split}
    \end{equation}

    Using the rescaling exploited in Remark~\ref{rem:scaling-wasserstein} we see that
        \[
            \sup\left\{\W_p^p(\rho)\colon\rho\in\AA_m\right\} = m^{1+\frac{p}{N}}\sup\left\{\W_p^p(\rho)\colon\rho\in\AA\right\};
        \]
     hence, \eqref{eq:dichotomy-estimate} implies that
    \[
        \W_p^p(\rho_n)\leq C_{N,p}\eps+\left((m^1_n)^{1+\frac{p}{N}}+(m^2_n)^{1+\frac{p}{N}}\right)\sup\left\{\W_p^p(\rho)\colon\rho\in\AA\right\}.
    \]
    If $\eps$ is small enough, this is incompatible with the fact that $\lim_n\W_p(\rho_n) = \sup_{\rho\in\AA}\W_p(\rho)$. In fact, the function $t\mapsto t^{1+\frac{p}{N}}$ is strictly convex, and if $\eps<\frac12\min\{\lambda,1-\lambda\}$, then $m^1_n$ and $m^2_n$ are far away from $0$.
\end{proof}

\bigskip

\subsection{The only maximizer is the ball}
In the second part of this section we will characterize the maximizers of $\W_p$ over $\AA$. In fact, we prove that the only maximizer of $\W_p$ is the characteristic function of a ball (with the correct volume). The intuition behind this result is that, if we have a set, and we create some holes in it (adding some mass somewhere else), we are lowering the energy since the additional mass can be transported at shorter distance. We obtain the main result in several steps: First we study the $1$-dimensional case, possibly with a weight, where the structure of the transport plan is known explicitly. Then, using a symmetrization argument we show that the optimal plan associated to a maximizer has some geometric properties, and, in fact, it is radial. Next, using the $1$-dimensional case, we prove that a maximizer has to be a star-shaped set, and via an optimality argument we deduce that a star-shaped maximizer must actually be a ball.

\begin{prop}\label{prop:maximizer-1D}
    Let $m>0$ be a given parameter. Let $w\colon(0,+\infty)\to(0,+\infty)$ be a non-decreasing weight and let $I=[0,\ell]$ be the unique segment such that $\int_I w\,dx=m$. For any density $\rho\colon\R^+\to [0,1]$ with $\int_{\R^+} \rho w\,dx = m$, we have that
    \begin{equation}\label{eq:1D-ineq}
        \W_p(\Chi{I})\geq \W_p(\rho),
    \end{equation}
    where $\W_p$ is defined in the metric-measure setting with base space $\R^+$ endowed with the usual distance and reference measure equal to $w\L^1$. Equality holds if and only if $\rho = \Chi{I}$ almost everywhere.
\end{prop}

\begin{proof}
    We note that, also in this weighted case, the transport distance is bounded (using again Lemma~\ref{lem:full-ball}), and thus for any density the infimum in the definition of $\W_p$ is achieved thanks to Theorem~\ref{thm:existence-uniqueness-transport} and Theorem~\ref{thm:existence-W_p}. Therefore, there exists $\eta_{\rho}$ such that $\W_p(\rho) = W_p(\rho\gamma,\eta_{\rho}\gamma)$, where we use the notation $\gamma=w\L^1$. Moreover, since the cost increases with the distance, we have that $\W_p(\Chi{I}) = W_p(\Chi{I}\gamma,\Chi{I'}\gamma)$, where $I' = [\ell,\ell']$ for some $\ell'>\ell$, and the transport plan is induced by a monotone map $T$ (see Theorem~\ref{thm:transport-1D}).
    
    Now we introduce an auxiliary problem that produces a non-optimal candidate to estimate $\W_p(\rho)$. The advantage of this modified problem is that it enforces a geometric constraint that clarifies some arguments. The auxiliary functional, which considers only plans which move mass to the right, is given by
    \begin{equation*}
        \AW_p^p(\rho)\coloneqq \inf\left\{\int |x-y|^p\,d\pi(x,y)\colon \pi\in\AP_{\rho}, \ \pi\big(\left\{(x,y)\colon y<x\right\}\big) = 0\right\}.
    \end{equation*}
    
    We observe that the infimum is actually a minimum since the additional constraint is closed under weak$*$ convergence. Moreover, applying the standard results for the one dimensional transport problem, we know that the optimal plan is induced by a non-decreasing map. Since we have already observed that $\W_p(\Chi{I}) = W_p(\Chi{I}\gamma,\Chi{I'}\gamma)$, the monotonicity of the optimal map ensures that $\AW_p(\Chi{I})=\W_p(\Chi{I})$. For a general density $\rho$, on the other hand, we have the inequality $\AW_p(\rho)\geq \W_p(\rho)$ due to the introduction of the additional constraint. With these observations, we reduce to proving the following (stronger) inequality:
    \[
        \W_p(\Chi{I})\geq \AW_p(\rho),
    \]
   and \eqref{eq:1D-ineq} simply follows.
 
   From now on we denote by $\tilde T_{\rho}$ the transport map appearing when we compute $\AW_p(\rho)$.
    We define the following ``volume'' functions on  $\R^+$:
    \[
        V(x) \coloneqq \int_0^xw(t)\,dt, \qquad V_{\rho}(x) \coloneqq \int_0^x\rho(t)w(t)\,dt.
    \]
    We also denote by $d(v)$ (resp. $d_{\rho}(v)$) the transport distance of the point $V^{-1}(v)$ (resp. $V_{\rho}^{-1}(v)$) when we compute $\W_p(\Chi{I})$ (resp. $\W_p(\rho)$), i.e.
    \begin{equation}\label{eq:d}
        d(v)\coloneqq |T(V^{-1}(v))-V^{-1}(v)|,\qquad d_{\rho}(v)\coloneqq |\tilde T_{\rho}(V_{\rho}^{-1}(v))-V_{\rho}^{-1}(v)|.
    \end{equation}
    Using the explicit expression of the optimal transport map in $1$D (see for example \cite[Remarks 2.19 (iv)]{V2003}), we have that
    \[
        \gamma([V^{-1}(v),V^{-1}(v)+d(v)]) = m\qquad \forall v\in[0,m].
    \]
    One can easily adapt the proof of Lemma~\ref{lem:full-ball} to the auxiliary functional and see that, if $x$ is a Lebesgue point for $\tilde T_{\rho}$ and $r=|\tilde T_{\rho}(x)-x|$, then $(\tilde T_{\rho})_{\#}(\rho\gamma) = (1-\rho)\gamma$ in $[x,x+r]$. Moreover, since $\tilde T_{\rho}$ is non-decreasing, we also have that
    \begin{equation}\label{eq:constraint-saturation}
        (\tilde T_{\rho})_{\#}\left(\rho\gamma\res[0,x]\right) = (1-\rho)\gamma\qquad \text{in }[x,x+r].
    \end{equation}
    
    We claim that $d_{\rho}\leq d$. In fact, suppose for contradiction that there exists $v\in(0,m)$ such that $d_{\rho}(v)>d(v)$. Since $\rho\leq1$ we have $V_{\rho}^{-1}\ge V^{-1}\ge 0$, and thus 
    \begin{equation*}
    \begin{split}
        \int_0^{V_{\rho}^{-1}(v)+d_{\rho}(v)}\rho(t)w(t)\,dt 
       &\geq \int_{V^{-1}(v)}^{V^{-1}(v)+d_{\rho}(v)}w(t)\,dt\\
        & > \int_{V^{-1}(v)}^{V^{-1}(v)+d(v)}w(t)\,dt = \gamma([V^{-1}(v),V^{-1}(v)+d(v)]) = m,
    \end{split}
    \end{equation*}
    where we have used \eqref{eq:constraint-saturation} with $x=V_{\rho}^{-1}(v)$ and $r=d_{\rho}(v)$ to get the first inequality, and the monotonicity of $w$ to obtain the second one. This chain of inequalities of course leads to a contradiction since $m=\int_0^\infty \rho \,d\gamma$. Therefore $d_{\rho}\leq d$.

    Since $w$ and $\rho w$ are locally bounded in $[0,+\infty)$, then both $V$ and $V_{\rho}$ are locally Lipschitz, and we can apply the fundamental theorem of calculus: using $v=V_{\rho}(x)$ as variable in the computation of $\AW_p(\rho)$ we obtain that
    \begin{equation*}
    \begin{split}
    	\AW_p^p(\rho) = \int_{\R^+} |\tilde T_{\rho}(x)-x|^p \rho(x)w(x)\,dx = \int_0^m d_{\rho}(v)^p\,dv\leq \int_0^md(v)^p\,dv = \W_p^p(\Chi{I}),
    \end{split}
    \end{equation*}
    where the inequality follows from comparison between $d$ and $d_{\rho}$, and this is the desired inequality. Finally, one can notice that the only way to obtain an equality in the previous chain of inequalities is that $\rho= \Chi{I''}$ for some segment $I''$ and $w$ is constant in $\spt (\rho+T_{\#}\rho)$. However, if $I''\neq I$, then one can construct a better transport plan that moves some mass to the left (this plan should belong to $\AP_{\rho}$, but it is not admissible for the auxiliary problem). Therefore, the equality in \eqref{eq:1D-ineq} holds only for $\rho=\Chi{I}$.
\end{proof}

\begin{lemma}\label{lemma:non-crossing}
    Let $p>1$ be given, and let $\rho\in\AA$ be a maximizer of $\W_p$. If $\nu\in\S^{N-1}$ is such that
    \begin{equation}\label{eq:half-direction}
        \int_{\{x\colon\scal{x}{\nu}>0\}}\rho \,dx = \int_{\{x\colon\scal{x}{\nu}<0\}}\rho \,dx = \frac12,
    \end{equation}
    then the optimal plan $\pi_{\rho}$ satisfies
    \begin{equation}\label{eq:splitting-plan}
        \pi_{\rho}(\{(x,y)\colon\scal{x}{\nu}\cdot\scal{y}{\nu}<0\})=0.
    \end{equation}
\end{lemma}

\begin{proof}
    The idea is to consider an auxiliary functional, as in the proof of Proposition~\ref{prop:maximizer-1D}, and show that it coincides with $\W_p$ when evaluated at $\rho$ (due to the maximality of this density). This ensures that $\pi_{\rho}$ has some additional structure due to the uniqueness of the optimal plan. 
    
    We define the auxiliary functional
    \[
        \AW_{p}^p(\rho,\nu)\coloneqq \inf\left\{\int |x-y|^p\,d\pi(x,y)\colon\pi\in\AP_{\rho}, \pi(\{(x,y)\colon\scal{x}{\nu}\cdot\scal{y}{\nu}<0\})=0\right\}.
    \]
    Loosely speaking, this auxiliary functional uses only plans that do not transport mass across the hyperplane $\{x\colon\scal{x}{\nu}=0\}$. As before, we are introducing an additional constraint that is closed under weak$*$ convergence, and thus there exists an optimal plan in the definition of $\AW_p(\rho,\nu)$. Clearly, since we are introducing a constraint in the minimization process, we have that $\W_p(\rho)\leq \AW_p(\rho,\nu)$. 
    
    Let $F(x)= x-2\scal{x}{\nu}\nu$ be the reflection map, and define the two symmetrizations of $\rho$ with respect to $\nu$:
    \[
        \rho_1 = \rho\res H_1+F_{\#}(\rho\res H_1), \qquad \rho_2 = \rho\res H_2+F_{\#}(\rho\res H_2),
    \]
    where $H_1 = \{x\colon\scal{x}{\nu}>0\}$ and $H_2 = F(H_1) = \{x\colon\scal{x}{\nu}<0\}$. We denote by $\bar \pi_{1}$ and $\bar \pi_{2}$ the two optimal plans realizing $\AW_p(\rho_1,\nu)$ and $\AW_p(\rho_2,\nu)$, respectively. We claim that 
    \[
        \bar \pi = \bar \pi_1\res(H_1\times H_1)+ \bar \pi_2\res(H_2\times H_2)
    \]
    realizes $\AW_p(\rho,\nu)$. In fact, $\bar \pi$ is admissible to compute $\AW_p(\rho,\nu)$, and if we find a better candidate $\pi$ to compute $\AW_p(\rho,\nu)$, then we can also construct the following plans that are good candidates to compute $\AW_p(\rho_1,\nu)$ and $\AW_p(\rho_2,\nu)$ respectively:
    \[
        \pi_1 = \pi\res(H_1\times H_1)+\tilde F_{\#}(\pi\res(H_1\times H_1)),\qquad \pi_2 = \pi\res(H_2\times H_2)+\tilde F_{\#}(\pi\res(H_2\times H_2)),
    \]
    where $\tilde F(x,y) = (F(x),F(y))$. Then we observe that
    \begin{gather*}
        \AW_p^p(\rho_1,\nu) = \int |x-y|^p\,d\bar\pi_1 = 2 \int_{H_1\times H_1} |x-y|^p\,d\bar\pi_1,\\
        \AW_p^p(\rho_2,\nu) = \int |x-y|^p\,d\bar\pi_2 = 2 \int_{H_2\times H_2} |x-y|^p\,d\bar\pi_2,\\
        \int |x-y|^p\,d\bar \pi = \frac12\left(\AW_p^p(\rho_1,\nu)+\AW_p^p(\rho_2,\nu)\right),\\
        \int |x-y|^p\,d \pi = \frac12\left(\int |x-y|^p\,d \pi_1+\int |x-y|^p\,d \pi_2\right).
    \end{gather*}
    If $\AW_p^p(\rho,\nu)<\int|x-y|^p\,d\bar\pi$, then at least one between $\pi_1$ and $\pi_2$ is a better competitor for $\AW_p(\rho_1,\nu)$ or $\AW_p(\rho_2,\nu)$, contradicting the definition of $\bar \pi_1$ and $\bar \pi_2$. Therefore, the following conditions hold:
    \begin{equation}\label{eq:symmetrization-maximizer}
        \W_p^p(\rho)\leq \AW_p^p(\rho,\nu) = \frac12\left(\AW_p^p(\rho_1,\nu)+\AW_p^p(\rho_2,\nu)\right) = \frac12\left(\W_p^p(\rho_1)+\W_p^p(\rho_2)\right),
    \end{equation}
    where we used the second part of Lemma~\ref{lemma:symmetry-transport} to obtain the last equality. Since $\rho$ is a maximizer, then \eqref{eq:symmetrization-maximizer} guarantees that $\rho_1$ and $\rho_2$ are also maximizers. This, however, implies that $\W_p(\rho) = \AW_p(\rho,\nu)$. In other words, $\bar \pi$ realizes $\W_p(\rho)$ and satisfies \eqref{eq:splitting-plan}. Therefore, necessarily, we have that $\pi_{\rho}=\bar \pi$, concluding the proof. 
\end{proof}

\medskip

\begin{corol}\label{cor:radial-transport-for-maximizers}
    Let $p>1$ be given, and let $\rho\in\AA$ be a maximizer of $\W_p$. Then there exists $x_0\in\R^N$ such that $\pi_{\rho}$ has the following property:
    \begin{equation}\label{eq:radial-plan}
        \pi_{\rho}\left(\{(x,y)\colon \scal{y-x_0}{x-x_0}\neq |y-x_0||x-x_0|\}\right) = 0.
    \end{equation}
    That is, $\pi_{\rho}$ is radial with center $x_0$.
\end{corol}

\begin{proof}
    By sliding each hyperplane $\{x\colon\scal{x}{e_i}=0\}$ until it splits the mass of $\rho$ in half, and by taking the intersection of the $N$ hyperplanes, we find a point $x_0\in\R^N$ such that
    \[
        \int_{\{x\colon\scal{x-x_0}{e_i}>0\}}\rho \,dx = \int_{\{x\colon\scal{x-x_0}{e_i}<0\}}\rho \,dx = \frac12\qquad \forall i\in\{1,\ldots,N\}.
    \]
    Up to translations, we suppose that $x_0=0$. By \eqref{eq:symmetrization-maximizer} we know that suitable symmetrizations of $\rho$ with respect to the coordinate axes are again maximizers. Iterating this procedure, we obtain a maximizer $\tilde \rho$ taking successive reflections of the sector
    \begin{equation}\label{eq:sector-rho}
        \rho\res\{x\colon\scal{x}{e_i}>0\ \forall i=1,\ldots,N\},
    \end{equation}
    and the result is a density symmetric with respect to each coordinate direction. 
    The symmetries of $\tilde \rho$ guarantee that
    \[
        \int_{\{x\colon\scal{x}{\nu}>0\}}\tilde{\rho} \,dx = \int_{\{x\colon\scal{x}{\nu}<0\}}\tilde{\rho} \,dx = \frac12\qquad \forall \nu\in\S^{N-1}.
    \]
    Hence, applying Lemma~\ref{lemma:non-crossing} to $\tilde \rho$ we obtain that $\pi_{\tilde \rho}$ satisfies the splitting condition \eqref{eq:splitting-plan} for any vector $\nu$. Thus, the condition \eqref{eq:radial-plan} holds for $\pi_{\tilde\rho}$. We finally conclude by uniqueness of the optimal plan, as we did in the last part of Lemma~\ref{lemma:non-crossing}: we can use the same strategy starting from a different sector in \eqref{eq:sector-rho}, defining a different symmetric density $\tilde \rho$. The same conclusion holds for the new optimal plan associated to that density, namely $\pi_{\tilde \rho}$. By uniqueness of the optimal plan, we know that $\pi_{\rho}$ can be obtained gluing together the plans of each sector, and thus also $\pi_{\rho}$ satisfies \eqref{eq:radial-plan}.
\end{proof}

\medskip

Now we can state and prove our main result.

\begin{theorem}\label{thm:ball-max}
    Let $p>1$ be given. Then the only maximizer of $\W_p$ in the class $\AA$, up to translations, is the characteristic function of $B$ with $|B|=1$.
\end{theorem}

\begin{proof}
    We prove this result in two steps: we first show that any maximizer must be the characteristic function of a star-shaped set, and then  exploit the inner-ball condition exposed in Lemma~\ref{lem:full-ball} to see that the length of the rays must be constant. Without loss of generality, we can suppose $N\geq 2$ since the $1$-dimensional case has already been treated in Proposition~\ref{prop:maximizer-1D}.

    \medskip
    
    \noindent \textit{Step 1.} First we will apply Corollary~\ref{cor:radial-transport-for-maximizers} and decompose the transport along rays. Then, we exploit the one dimensional result obtained in Proposition~\ref{prop:maximizer-1D} to prove that the maximizer intersects each
    ray in a segment emanating from the origin.

    Let $\rho$ be any maximizer of $\W_p$ in $\AA$. We apply Corollary~\ref{cor:radial-transport-for-maximizers} to $\rho$, and suppose, without loss of generality, that the point $x_0$ coincides with the origin. Therefore, the optimal plan $\pi_{\rho}$ is induced by a radial map $T_{\rho}$. Since in this proof we do not need to stress the dependence of $\eta_{\rho}$, $\pi_{\rho}$ and $T_{\rho}$ on the density $\rho$, we simplify the notation, and we denote those objects by $\eta$, $\pi$ and $T$, respectively. We decompose every function in radial coordinates, and let $w(r) = r^{N-1}$ denote the coarea factor when we integrate in polar coordinates. For any $\omega\in\S^{N-1}$ we define the functions
    \[
        \rho^{\omega}(r) = \rho(r\omega),\qquad \eta^{\omega}(r) = \eta(r\omega),\qquad T^{\omega}(r) = |T(r\omega)|
    \]
    for every $r\in[0,+\infty)$. We consider them as functions defined (almost everywhere) on the metric-measure space $(X,d,\gamma)$, where $X=\R^+$, $\gamma=w\L^1$ and $d$ is the usual distance.
    
    We claim that, since $T(r\omega) = T^{\omega}(r)\omega$ and $T_{\#}\rho = \eta$, we have
    \begin{equation}\label{eq:factorization-transport}
        (T^{\omega})_{\#}\left(\rho^{\omega}\gamma\right) = \eta^{\omega}\gamma\qquad \text{for a.e. } \omega\in\S^{N-1}.
    \end{equation}
    For any $s>0$ and any $E\subset \S^{N-1}$ we define the set $F = \{r\omega\colon0\leq r\leq s,\, \omega\in E\}$ and we have that
    \begin{equation*}
    \begin{split}
        \int_Ed\H^{N-1}_{\omega}\int_0^s\eta^{\omega}\,d\gamma &= \int_Ed\H^{N-1}_{\omega}\int_0^s \eta(r\omega)r^{N-1}\,dr = \int_F\eta(x)\,dx\\
        &=\int_F(T_{\#}\rho)(x)\,dx =\int_{T^{-1}(F)}\rho(x)\,dx\\
        &=\int_Ed\H^{N-1}_{\omega}\int_{(T^{\omega})^{-1}([0,s])}\rho(r\omega)r^{N-1}\,dr\\
        &=\int_Ed\H^{N-1}_{\omega}\int_{(T^{\omega})^{-1}([0,s])}\rho^{\omega}(r)\,d\gamma\\
        &=\int_Ed\H^{N-1}_{\omega}\int_0^s(T^{\omega})_{\#}\rho^{\omega}\,d\gamma.
    \end{split}
    \end{equation*}
    Here we used that $T$ is radial to pass from the second to the third line, in combination with the integration in polar coordinates. Since $E$ and $s$ are arbitrary, this proves \eqref{eq:factorization-transport}.
    
    We obtain the result of this first step by applying Proposition~\ref{prop:maximizer-1D} separately for any $\omega\in\S^{N-1}$. In fact, we can integrate in polar coordinates the transport cost and obtain that
    \begin{align*}
        \int |T(x)-x|^p\rho(x)\,dx &= \int_{\S^{N-1}}\int_0^{+\infty} |T(r\omega)-r\omega|^p\rho(r\omega)r^{N-1}\,dr\,d\omega\\
        &= \int_{\S^{N-1}}\left(\int_0^{+\infty}|T^{\omega}(r)-r|^p\rho^{\omega}(r)w(r)\,dr\right)\,d\omega.
    \end{align*}
    The inner integral in the last expression coincides with the transport cost of $T^{\omega}$ between $\rho^{\omega}\gamma$ and $\eta^{\omega}\gamma$, and since $T$ is the optimal transport map between $\rho$ and $\eta$, then also $T^{\omega}$ must be optimal between $\rho^{\omega}\gamma$ and $\eta^{\omega}\gamma$ for every $\omega\in\S^{N-1}$. This is properly justified by showing that gluing the optimizers $\omega$-by-$\omega$ we obtain a measurable density. We sketch the proof of this fact in Appendix~\ref{sec:appendix}.
     Therefore, if we denote by $m(\omega) = \int_{\R^+} \rho^{\omega}\,d\gamma$, then
    \begin{align}\label{eq:radial-optimal-decomposition}
        \W_p^p(\rho) &= \int_{\S^{N-1}}\W_p^p(\rho^{\omega})\,d\omega \notag\\ 
        &\leq \int_{\S^{N-1}} \sup\left\{\W_p^p(\theta)\colon \theta\colon X\to[0,1],\int_{X}\theta \,d\gamma=m(\omega)\right\}\,d\omega,
    \end{align}
    where we use the metric-measure definition of $\W_p$ in those integrals (see Remark~\ref{rem:metric-measure-setting}).  By Proposition~\ref{prop:maximizer-1D}, for every $\omega\in\S^{N-1}$, the supremum inside the last integral coincides with $\W_p^p(\Chi{I^{\omega}})$, where $I^{\omega}\subset X$ is the unique segment of the form $[0,\ell^{\omega}]$ with $\gamma(I^{\omega}) = m(\omega)$. Moreover, the inequality is strict whenever $\rho^{\omega}$ is not equivalent to $\Chi{I^{\omega}}$. Since the map $\omega\mapsto m(\omega)$ is measurable, we can glue the segments $I^{\omega}$ together and obtain another candidate to compute $\W_p$. The density $\rho$ is a maximizer; hence, for almost every $\omega\in\S^{N-1}$ the density $\rho^{\omega}$ must be equivalent to $\Chi{I^{\omega}}$, concluding the proof of the first step.

    \medskip
    
    \noindent \textit{Step 2.} For any $\omega\in\S^{N-1}$ we know that $T(\ell^{\omega}\omega) =(T^{\omega}(\ell^{\omega}))\omega$, and Lemma~\ref{lem:full-ball} guarantees that $\eta(x) = 1-\rho(x)$ for every $x\in\R^N$ such that $|x-\ell^{\omega}\omega|\leq T^{\omega}(\ell^{\omega})-\ell^{\omega}$. Let $\nu\in\S^{N-1}$ be another unit vector. 
    Note that $T^{\omega}(\ell^{\omega}) = 2^{1/N}\ell^{\omega}$ (see e.g. \cite{C-TG2022} where the transport map in the case of a ball is given explicitly). Thanks to the inner ball condition, we obtain that $T^{\nu}(\ell^{\nu})$ is larger than $t$ for any $t>0$ such that $|t\nu-\ell^{\omega}\omega|\leq T^{\omega}(\ell^{\omega})-\ell^{\omega}$. 
    
    In order to simplify the notation we define $c=2^{1/N}$, $r=\ell^{\omega}$ and $s=\ell^{\nu}$. Taking the square of both sides of the inner-ball inequality (see Figure~\ref{fig:star-shaped-comparison} for a geometric intuition of the inner ball condition in this situation), we get that $s\geq t$ for every $t>0$ satisfying
    \begin{equation*}
        c^2t^2-2c\scal{\nu}{\omega}rt+c(2-c)r^2= 0.
    \end{equation*}
    
    Solving the above equation in $t$ one gets that
    \[
        s\geq \frac{\scal{\nu}{\omega}+\sqrt{ \scal{\nu}{\omega}^{2}-c(2-c) }}{c}\,r \, .
    \]
    By the definition of $c$, the expression under the square root is non-negative whenever $\scal{\nu}{\omega}$ is close enough to $1$ since $c>1$ and $1-c(2-c) = (c-1)^2>0$. Swapping the roles of $\nu$ and $\omega$ we also arrive to the analogous inequality
    \[
        r\geq \frac{\scal{\nu}{\omega}+\sqrt{ \scal{\nu}{\omega}^{2}-c(2-c) }}{c}\,s \, .
    \]
    \begin{figure}[t!]
        \centering
        \includegraphics[width=0.3\textwidth]{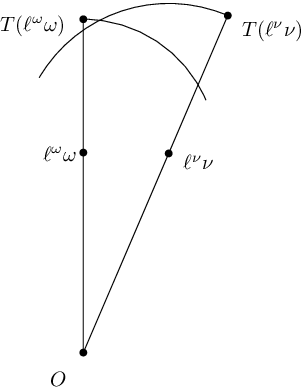}
        \caption{In this figure we depict two points $\ell^{\omega}\omega$ and $\ell^{\nu}\nu$ that belong to the support of $\rho$, and their images through the map $T$, which coincide with $T^{\omega}(\ell^{\omega})\omega$ and $T^{\nu}(\ell^{\nu})\nu$ respectively. The inner ball condition implies that the two image points have to lie outside the circles centered at these points with radii given by the transport distances $T^{\omega}(\ell^{\omega})-\ell^{\omega}$ and $T^{\nu}(\ell^{\nu})-\ell^{\nu}$ respectively.}\label{fig:star-shaped-comparison}
    \end{figure}
    
    Combining these two inequalities we can control the difference between $s$ and $r$ in terms of the distance between $\nu$ and $\omega$:
    \begin{align*}
        s-r&\geq \frac{r}{c}\left( \scal{\nu}{\omega}-c+\sqrt{ \scal{\nu}{\omega}^2-c(2-c) } \right) 
	= \frac{2r(1-\scal{\nu}{\omega})}{\scal{\nu}{\omega}-c-\sqrt{ \scal{\nu}{\omega}^2-c(2-c) }},\\
        s-r&\leq \frac{s}{c}\left(c-\scal{\nu}{\omega}-\sqrt{ \scal{\nu}{\omega}^2-c(2-c) }\right)
	=\frac{2s(1-\scal{\nu}{\omega})}{c-\scal{\nu}{\omega}+\sqrt{ \scal{\nu}{\omega}^{2}+c(2-c) }}.
    \end{align*}
   By Corollary~\ref{cor:bounded-transport-distance} we have that $|T(x)-x|\leq C_N$ for a dimensional constant $C_N$; hence, $r$ and $s$ are also uniformly bounded. Since $2(1-\scal{\nu}{\omega}) = |\nu-\omega|^2$, we can combine the previous estimates and obtain that
    \[
        |\ell^{\nu}-\ell^{\omega}|\leq \tilde C_N |\nu-\omega|^2
    \]
    for any $\nu$ and $\omega$ sufficiently close. This implies that the map $\omega\mapsto \ell^{\omega}$ is $2$-H\"{o}lder continuous on the sphere, hence it is constant. This is equivalent to showing that the only maximizer is the ball, and thus the proof is concluded.
\end{proof}

\medskip
\section{Quantitative inequality in one dimension}\label{sec:quant_ineq}

In this section we prove a quantitative inequality for $\W_p$ in one dimension, so we manage to strengthen the result obtained in Section~\ref{sec:maximizer} adding a term that measures the displacement of a density $\rho$ respect to the characteristic function of a ball. In order to measure that distance, we consider a version of the Frankel asymmetry that, loosely speaking, is the $L^1$ distance between a density and a ball. This choice is by no means new: for example, the asymmetry was used in the quantitative isoperimetric inequality (cfr. \cite{FMP2008,FMP2010}) and in the quantitative Brunn-Minkowski inequality (cfr. \cite{BJ2017}). See also \cite{FP2020,FL2021} for a quantitative inequality involving a functional of Riesz type.

\begin{defin}
    We define the following quantity, that we will just call asymmetry in the sequel:
    \begin{equation*}
        A(\rho)\coloneqq \inf\left\{\norm{\rho-\Chi{B_r(x)}}_1\colon x\in\R^N,|B_r(x)|=1\right\}\qquad \forall \rho\in\AA.
    \end{equation*}
\end{defin}

With this notion, our quantitative inequality reads as the following.

\medskip

\begin{theorem}\label{thm:quantitative-inequality-1D}
    For $N=1$ and $p>1$ fixed, there exists a constant $C_p>0$ such that
    \[
        \W_p^p(B)-\W_p^p(\rho)\geq C_pA(\rho)^2\qquad \text{for all } \rho\in\AA.
    \]
\end{theorem}

\medskip

\begin{remark}
    We point out that the exponent $2$ in our quantitative inequality is sharp, in the sense that the inequality would be false with a smaller exponent for densities with small asymmetry. This can be seen by taking $\rho = \Chi{[-1/2-\eps,-1/2]}+\Chi{[-1/2+\eps,1/2-\eps]}+\Chi{[1/2,1/2+\eps]}$ for $\eps$ small. Notice that $\rho$ is symmetric, and using Lemma~\ref{lemma:non-crossing} we can restrict to work in $\R^+$. Moreover, Theorem~\ref{thm:transport-1D} guarantees that the optimal transport map is monotone, and this allows us to compute the transport map $T_{\rho}$. In fact, we claim that the transport map to compute $\W_p(\rho)$ has the following expression for every $x\in\R^+$:
    \begin{equation}\label{eq:sharp-map}
    T_{\rho}(x) = \begin{cases}
        x+\frac{1}{2}-\eps &\text{for }x\in\left(0,\eps\right),\\
        x+\frac{1}{2} & \text{for }x\in\left(\eps,\frac{1}{2}-\eps\right),\\
        x+\frac{1}{2}-\eps &\text{for }x\in\left(\frac{1}{2},\frac{1}{2}+\eps\right).
    \end{cases}
    \end{equation}
    We prove that the optimal transport map coincides with the above expression just for $x\in(0,\eps)$, the other cases being analogous. Suppose that $T_{\rho}(z)=y<z+\frac{1}{2}-\eps$ for some $z\in(0,\eps)$. Since $T_{\rho}$ is monotone increasing in that interval, $T_{\rho}\res (0,z)\leq y$. The $L^{\infty}$-constraint is already saturated in $(0,\frac{1}{2}-\eps)$, and thus we have that
    \[
        \frac{1}{2}-\eps\leq T_{\rho}\res(0,z)\leq y< z+\frac{1}{2}-\eps.
    \]
    But this is not possible since $y-(\frac{1}{2}-\eps) < z = \int_0^z\rho = (T_{\rho})_{\#}(\rho\res(0,z))\leq y-(\frac{1}{2}-\eps)$. This proves that $T_{\rho}$ is pointwise larger or equal than the expression in our claim. However, any map that is strictly larger than the function in \eqref{eq:sharp-map} has also strictly larger transport cost, and thus it is not optimal. With the explicit expression of $T_{\rho}$ the conclusion follows from an easy computation:
    \[
    \begin{split}
        \W_p^p(\rho) &= 2\int_{\R^+}|T_{\rho}(x)-x|^p\rho(x)\,dx = 2\left[\left(\frac{1}{2}-\eps\right)^p\eps + \frac{1}{2^p}\left(\frac{1}{2}-2\eps\right)+\left(\frac{1}{2}-\eps\right)^p\eps\right]\\
        &=\frac{1}{2^p}-\frac{4\eps}{2^p}+4\eps\left(\frac{1}{2}-\eps\right)^p = \W_p^p\left(\Chi{B}\right) - 4\eps \left[\frac{1}{2^p}-\left(\frac{1}{2}-\eps\right)^p\right],
    \end{split}
    \]
    where we used the explicit value of the energy of the ball $B=(-1/2,1/2)$. We thus conclude, since the final expression inside the square parentheses is $O(\eps)$.
\end{remark}

\bigskip

\begin{proof}
    
    By definition of asymmetry, $A(\rho)\leq 2$ for every $\rho\in\AA$, and without loss of generality we can suppose that $A(\rho)>0$. Up to translations, we can suppose that
    \[
        \int_{-\infty}^{0}\rho \,dx = \int_{0}^{+\infty}\rho \,dx = \frac12.
    \]
    Notice that, using the construction in Proposition~\ref{prop:maximizer-1D}, with constant weight $w\equiv 1$, we can produce a transport plan $\bar\pi\in\AP_{\rho}$ with $|x-y|\leq 1/2$ for any $(x,y)\in\spt\bar\pi$. In fact, when the weight is constant, the function $d$ defined in \eqref{eq:d} is constantly equal to $m$ (corresponding to the parameter in the statement of the proposition). Along our argument in Proposition~\ref{prop:maximizer-1D} we show that $d_{\rho}\leq d$, and in the present situation we have that the function $d$ is constantly equal to $1/2$ since $\int_{\R^+}\rho = \int_{\R^+}\rho w=1/2$. This is actually equivalent to saying that $|x-y|\leq 1/2$ for any $(x,y)\in\spt\bar\pi$, since that particular transport plan is induced by the map $\tilde T_{\rho}$ defined in that proposition. 
    Loosely speaking, $\bar \pi$ moves mass ``away from the origin''. Now we want to get a quantitative inequality modifying $\bar\pi$ and finding another plan $\pi\in\AP_{\rho}$ for which the transport distance is again bounded from above by $1/2$ in a pointwise sense, and moreover
    \begin{equation}\label{eq:measure-distance-estimate}
        \pi\left(\left\{(x,y)\colon|x-y|\leq d_A\right\}\right)\geq \frac{A(\rho)}{100},\quad \text{where } d_A\coloneqq \frac12-\frac{A(\rho)}{100}.
    \end{equation}
    With this competitor, if $E=\left\{(x,y)\in\R^N\times\R^N\colon|x-y|\leq d_A\right\}$ is the set considered in the previous inequality, we have that
    \begin{equation*}
    \begin{split}
        \W_p^p(\rho)&\leq \int|x-y|^p\,d\pi(x,y) \leq (d_A)^p\pi(E)+\frac{1}{2^p}(1-\pi(E))\\
        &= \frac{1}{2^p}+\frac{\pi(E)}{2^p}\left[\left(1-\frac{A(\rho)}{50}\right)^p-1\right]\leq \frac{1}{2^p}+\frac{\pi(E)}{2^p}\left(-C_pA(\rho)\right)\\
        &= \W_p^p(B)-C_pA(\rho)^2,
    \end{split}
    \end{equation*}
    where $C_p$ is a constant depending only on $p$. Therefore, we need to find such a plan $\pi$ to complete the proof. We denote by $\bar T$ the map that induces $\bar \pi$. Let us look at the set $\{x\geq0\}$, and we define $x_R$ as the smallest point that is moved at distance $d_A$, i.e. $x_R\coloneqq \inf\{x>0\colon\bar T(x)-x> d_A\}$. 
    Now we explore the different cases that may appear.

    \medskip
    
    \paragraph{\textit{Case 1}} If we have that $\int_0^{x_R}\rho \,dx\geq \frac{A(\rho)}{100}$, then the plan $\bar\pi$ already satisfies \eqref{eq:measure-distance-estimate} and there is nothing to do.

    \medskip

    \paragraph{\textit{Case 2}} Let us suppose that both of the following conditions hold
    \[
        \int_0^{x_R}\rho \,dx< \frac{A(\rho)}{100},\qquad \int_0^{x_R}(1-\rho) \,dx> \frac{A(\rho)}{100}.
    \]
    In this case, we take a point $x_R^1>x_R$ such that $\int_0^{x_R^1}\rho \,dx = \frac{A(\rho)}{100}$,  and we try to move mass in the opposite direction in the segment $[0,x_R^1]$. This is necessary in order to take into account densities similar to the characteristic function of the union of two intervals: in that case, the optimal map actually moves mass toward the origin (see Figure~\ref{fig:segments}). 

\bigskip
    
    \begin{figure}[h!]
    \centering
    \includegraphics[width=0.4\textwidth]{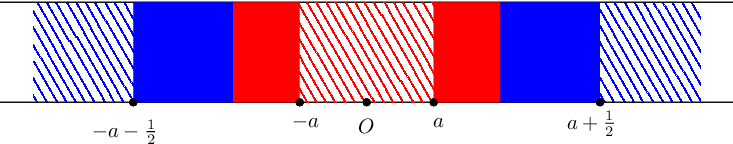}
    \caption{An optimal transport plan in dimension $N=1$ that moves some mass toward the origin. In this example, $\rho=\Chi{E}$ where $E= [-a-1/2,-a]\cup[a,a+1/2]$ for some small $a>0$ (shown in solid color). The optimal transport map sends the solid blue region to the shaded blue region, and the solid red region to the shaded red region. This map realizes $\W_p(\rho)$ for every $p\ge 1$. }\label{fig:segments}
\end{figure}

    To do this, we consider a transport plan tailored to $\rho$ and depending on $x_R^1$ that is obtained again through a minimization process:
    \begin{equation}\label{eq:constrained-transport-backward}
        \min\left\{\int |x-y|^p\,d\pi(x,y)\colon\pi\in\AP_{\rho}, \spt \pi\subset D\right\},
    \end{equation}
    where $D\subset \R\times\R$ is the following domain:
    \[
        D\coloneqq \{(x,y)\colon x\not\in(0,x_R^1), x\cdot(y-x)\geq 0\}\cup\left([0,x_R^1]\times[0,x_R^1]\right).
    \]
    Observe that, since $\int_0^{x_R^1}\rho \,dx=\frac{A(\rho)}{100}<\int_0^{x_R^1}(1-\rho)\,dx$, then it is possible to find a minimizer $\pi$ of \eqref{eq:constrained-transport-backward}. Applying again the structure theorem for optimal plans in one dimension, we find a map $T$ that induces an optimal plan. This transport problem is actually decoupled, considering independently $\rho\res [0,x_R^1]$ and $\rho-(\rho\res[0,x_R^1])$. Hence, it is possible to adapt \cite[Lemma 5.1]{DMSV2016} separately to both pieces and see that $|T(x)-x|\leq d_A$ for every $x\in[0,x_R^1]$. In fact, if this is not the case, then $T_{\#}\rho = 1-\rho$ in a segment $I\subset [0,x_R^1]$ longer than $d_A$. This is impossible since
    \[
        d_A\leq |I| = \int_I (\rho+(1-\rho))\,dx= \int_I(\rho+T_{\#}\rho)\,dx \leq 2\int_0^{x_R^1}\rho \,dx = 2\cdot \frac{A(\rho)}{100}\leq\frac{1}{25},
    \]
    and $d_A = \frac12-\frac{A(\rho)}{100}>\frac13$. Having this uniform bound on the transport length in $[0,x_R^1]$, then we see that $\pi$ satisfies \eqref{eq:measure-distance-estimate} because $\int_0^{x_R^1}\rho \,dx= \frac{A(\rho)}{100}$.

    \medskip

    \paragraph{\textit{Case 3}} Finally, let us suppose that the following inequalities hold at the same time:
    \[
        \int_0^{x_R}\rho \,dx< \frac{A(\rho)}{100},\qquad\int_0^{x_R}(1-\rho)\,dx\leq \frac{A(\rho)}{100}.
    \]
    At this point, we can explore each of the previous cases on the left side of the real line, producing the analogous $x_L = \sup\left\{x<0\colon x-\bar T(x)> d_A\right\}$. Since in the first two cases we managed to construct the desired $\pi$, we can suppose without loss of generality that we are in Case 3 also on the left side. In other words,  the following holds
    \[
        \max\left\{\int_0^{x_R}\rho \,dx, \int_0^{x_R}(1-\rho) \,dx, \int_{x_L}^0\rho \,dx, \int_{x_L}^0(1-\rho) \,dx \right\}\leq \frac{A(\rho)}{100}.
    \]
    Combining these information we obtain an estimate on $|x_R-x_L|$:
    \[
        x_R-x_L = \int_{x_L}^0(\rho+(1-\rho))\,dx + \int_0^{x_R}(\rho+(1-\rho))\,dx \leq \frac{A(\rho)}{25},
    \]
    and we will see that this is not possible because we can get an inequality for the asymmetry of $\rho$. We repeat here the argument of Proposition~\ref{prop:maximizer-1D}: adapting \cite[Lemma 5.1]{DMSV2016} we obtain that $\bar T_{\#}(\rho\res[x_L,0]) = 1-\rho$ in $[x_L-d_A,x_L]$ and $\bar T_{\#}(\rho\res[0,x_R]) = 1-\rho$ in $[x_R,x_R+d_A]$, and thus
    \begin{equation*} 
    \begin{split}
        \int_{x_L-d_A}^{x_R+d_A}\rho \,dx &= \int_{x_L-d_A}^{x_L}\rho \,dx+\int_{x_L}^0\rho \,dx+\int_0^{x_R}\rho \,dx + \int_{x_R}^{x_R+d_A}\rho \,dx\\
        & \geq \int_{x_L-d_A}^{x_L}\rho \,dx+ \int_{x_L-d_A}^{x_L}\bar T_{\#}\rho \,dx +\int_{x_R}^{x_R+d_A}\bar T_{\#}\rho \,dx+\int_{x_R}^{x_R+d_A}\rho\,dx=2d_A.
    \end{split}
    \end{equation*}
    This means that $\int_{x_L-d_A}^{x_R+d_A}\rho \geq 1-\frac{A(\rho)}{50}$. If $x_R+d_A-(x_L-d_A)\leq1$, then by definition of asymmetry
    \[
        A(\rho)\leq 2\int_{-\infty}^{x_L-d_A}\rho\,dx+2\int_{x_R+d_A}^{+\infty}\rho\,dx\leq \frac{A(\rho)}{25},
    \]
    that is impossible. Hence, we know that $x_L-d_A+1<x_R+d_A$. Since we proved that $x_R-x_L\leq \frac{A(\rho)}{25}$, we obtain an inequality always valid in our case: $x_R+d_A-(x_L-d_A) = 1-\frac{A(\rho)}{50}+x_R-x_L \leq 1+\frac{A(\rho)}{50}$. Therefore, we have that
    \begin{align*}
        A(\rho) &\leq 2\int_{x_L-d_A}^{x_L-d_A+1}(1-\rho)\,dx \leq 2\int_{x_L-d_A}^{x_R+d_A}(1-\rho)\,dx \leq 2(x_R-x_L+2d_A)-2\left(1-\frac{A(\rho)}{50}\right)\\
        &\leq 2+\frac{A(\rho)}{25}-2+\frac{A(\rho)}{25} = \frac{2}{25}A(\rho),
    \end{align*}
    and thus we reach a contradiction, concluding the last remaining case.
\end{proof}


\bigskip
\appendix
\section{Sketch of the measurability of the construction in Theorem~\ref{thm:ball-max}}\label{sec:appendix}

In Theorem~\ref{thm:ball-max} we needed to check that the density
\[
    (r,\omega) \mapsto \bar \zeta^{\omega}(r)
\]
is measurable, where $\bar \zeta^{\omega}$ satisfies $\W_p(\rho^{\omega}) = W_p(\rho^{\omega},\bar\zeta^{\omega})$. This is necessary to have the representation in \eqref{eq:radial-optimal-decomposition}. To do that, we approximate $\rho$ in $L^1$ with densities $\rho_k\in\AA$ that are piecewise constant along the sphere. In other words, for every $k$ there exists a partition of the sphere $\S^{N-1} = \bigcup_j E^k_j$ with sets such that $\diam (E^k_j)+|E^k_j|\leq 1/k$, and such that for every $j$
\[
    \rho_k(r\omega) = \rho_k(r\omega') \qquad\forall \omega,\omega'\in E^k_j.
\]
We construct the following densities: for every $k$ and every $\omega\in\S^{N-1}$ we take $\zeta$ such that $\W_p(\rho_k^{\omega}) = W_p(\rho_k^{\omega},\zeta)$ (in the metric-measure sense), and we define
\[
    \zeta_k(r,\omega) = \zeta(r).
\]
In other words, $\zeta_k^{\omega}$ is the optimal density to compute $\W_p(\rho_k^{\omega})$. This density is measurable since it is piecewise constant along the sphere. Since $\rho_k\to\rho$ in $L^1$, then $\rho_k^{\omega}\to \rho^{\omega}$ in $L^1$ for a.e. $\omega\in\S^{N-1}$. For this reason, we say that $\zeta_k^{\omega}\to \bar \zeta^{\omega}$ in weak$*$ sense for a.e. $\omega$.

To see this, notice that $\zeta_k^{\omega}$ converges to some density $\phi^{\omega}$ because the sequence $\rho_k^{\omega}$ is bounded in $L^{\infty}$, and the transport distance is bounded when the mass of $\rho_k^{\omega}$ is finite, that happens for a.e. $\omega$. By lower semicontinuity of the transport distance we have that
\[
    \W_p(\rho^{\omega}) = W_p(\rho^{\omega},\bar\zeta^{\omega}) \leq W_p(\rho^{\omega},\phi^{\omega}) \leq \liminf_k W_p(\rho_k^{\omega},\zeta_k^{\omega}) = \W_p(\rho^{\omega}),
\]
where we used that $\rho^{\omega}+\phi^{\omega}\leq1$ in the first inequality, and the continuity of $\W_p$ with respect the weak$*$ convergence in the last equality. Since the optimal density to compute $\W_p(\rho^{\omega})$ is unique, then $\bar \zeta^{\omega} = \phi^{\omega} = \lim_k \zeta_k^{\omega}$.

We finally conclude because $\zeta_k\to \zeta_\infty$ for some $\zeta_\infty$ in weak$*$ sense, and $\zeta_\infty$ is therefore measurable. Moreover, a little argument shows that, whenever $f_k:X\times Y\to\R$ converges in weak$*$ sense to $f$ ($X$ and $Y$ being reasonable spaces, in our case $X=\R^+$ and $Y=\S^{N-1}$), then for almost every $y\in Y$ we have that
\[
    f_k\res (X\times \{y\}) \weakstar f\res (X\times \{y\}).
\]
Hence, for almost every $\omega\in\S^{N-1}$ we have that
\[
    \zeta_k^{\omega}\weakstar \zeta_\infty^{\omega},
\]
and our previous argument shows also that
\[
    \zeta_k^{\omega}\weakstar \bar\zeta^{\omega}\qquad \text{for a.e. }\omega\in\S^{N-1}.
\]
Combining these facts, we get that $\bar\zeta = \zeta_\infty$ almost everywhere, and thus $\bar \zeta$ is measurable, as we wanted.

\bigskip
\subsection*{Acknowledgments}
We would like to thank Rupert Frank for pointing out this problem to us. D.C. is member of the Istituto Nazionale di Alta Matematica (INdAM), Gruppo Nazionale per l’Analisi Matematica, la Probabilità e le loro Applicazioni (GNAMPA), and is partially supported by the INdAM--GNAMPA 2023 Project \textit{Problemi variazionali per funzionali e operatori non-locali}, codice CUP\_E53\-C22\-001\-930\-001. D.C. wishes to thank Virginia Commonwealth University for the generous hospitality provided during his visit, which marked the beginning of this project. I.T.'s research was partially supported by a Simons Collaboration grant 851065 and an NSF grant DMS 2306962.

\bibliographystyle{IEEEtranSA}
\bibliography{references}

\end{document}